\documentclass[a4paper,10pt, oneside]{article}
\usepackage{amsmath, amssymb, amsthm,graphicx}
\usepackage[font=small,labelfont=md,textfont=it]{caption}
\usepackage{floatrow,float}
\usepackage[titletoc, title]{appendix}
\usepackage[colorlinks,linkcolor=blue,citecolor=blue]{hyperref}
\usepackage{longtable}
\usepackage{pgfplots}
\usepackage{diagbox}
\usepackage{booktabs,makecell,multirow}
\usepackage[capitalise, nosort]{cleveref}
\usepackage{cases,color}
\usepackage{algorithm}
\usepackage{algorithmic}
\usepackage{amsmath,bm}
\usepackage{geometry}
\usepackage{extarrows}
\usepackage{tcolorbox}

\DeclareMathOperator{\D}{D}

\crefname{equation}{}{}
\crefname{lem}{Lemma}{Lemmas}
\crefname{thm}{Theorem}{Theorems}

\newcommand{\bb}{{\mathbb B}}

\newcommand{\dd}{\,{\rm d}}
\newcommand{\R}{\,{\mathbb R}}

\newcommand{\dual}[1]{\left\langle {#1} \right\rangle}

\newcommand{\inner}[1]{\left({#1} \right)}

\newcommand{\nm}[1]{\left\lVert {#1} \right\rVert}
\newcommand{\snm}[1]{\left\lvert {#1} \right\rvert}
\newcommand{\ssnm}[1]
{
	\left\vert\kern-0.25ex
	\left\vert\kern-0.25ex
	\left\vert
	{#1}
	\right\vert\kern-0.25ex
	\right\vert\kern-0.25ex
	\right\vert
}

\makeatletter
\def\spher@harm#1{%
	\vbox{\hbox{%
			\offinterlineskip
			\valign{&\hb@xt@2\p@{\hss$##$\hss}\vskip.2ex\cr#1\crcr}%
		}\vskip-.36ex}%
}
\def\gshone{\spher@harm{.}}
\def\gshtwo{\spher@harm{.&.}}
\def\gshthree{\spher@harm{.&.&.}}
\let\gsh\spher@harm
\makeatother

\newtheorem{prop}{Proposition}[section]

\newtheorem{lem}{Lemma}[section]

\newtheorem{rem}{Remark}[section]

\newtheorem{thm}{Theorem}[section]

\newtheorem{eg}{Example}

\newcounter{mnote}
\let\oldmarginpar\marginpar
\renewcommand\marginpar[1]
{\-\oldmarginpar[\raggedleft\footnotesize #1]{\raggedright\footnotesize #1}}

\makeatletter\def\@captype{table}\makeatother

\newfloatcommand{capbtabbox}{table}[][\FBwidth]
\usepackage[T1]{fontenc} 
\usepackage[utf8]{inputenc} 
\usepackage{authblk}

\begin{document}
	\title{
		\Large \bf Optimal error estimation of a time-spectral method for 
		fractional diffusion problems with low regularity data
		\thanks
		{
			This work was supported in part by National Natural Science Foundation
			of China (11771312).
		}
	}
\author[,a,b]{Hao Luo\thanks{Email: luohao@math.pku.edu.cn}}
\author[,a]{Xiaoping Xie\thanks{Corresponding author. Email: xpxie@scu.edu.cn}}
\affil[a]{School of Mathematics, Sichuan University, Chengdu 610064, China} 
\affil[b]{School of Mathematical Sciences, Peking University, Beijing, 100871, China}
%
%	\author{
%		Hao Luo \thanks{School of Mathematical Sciences, 
%			Peking University, Beijing, 100871, China (Email: luohao@math.pku.edu.cn).}\thanks{School of Mathematics, Sichuan University, Chengdu 610064, China.},
%		Xiaoping Xie \thanks{. Corresponding author. Email: xpxie@scu.edu.cn}
%	}
	
	\date{}
	\maketitle
	%\tableofcontents
	
	\begin{abstract}
		This paper is devoted to the error analysis of a time-spectral algorithm 
		for fractional diffusion problems of order $\alpha$ ($0 < \alpha < 1$). The solution regularity in the Sobolev space is revisited, and  new regularity results in  the Besov space  are established. A time-spectral algorithm is developed which adopts a standard spectral method and a conforming linear finite element method for  temporal and spatial discretizations, respectively. Optimal error estimates are derived with nonsmooth data. Particularly, a sharp temporal convergence rate $1+2\alpha$ is shown theoretically and   numerically.
		%	We concern numerical analysis for the time
		%	fractional diffusion problems with nonsmooth data. 
		%	Regularity of the weak solution is firstly analyzed in Sobolev space and Besov space, respectively.  Then a time-stepping scheme is proposed and optimal convergence rate is derived with respect to the solution regularity in Sobolev space. Besides, a time-spectral algorithm is considered as well, and optimal error estimate is established in terms of the solution regularity in Besov space. Finally, several numerical tests are provided to validate the theoretical results.
		%	
		%	
		%  This paper presents an optimal error analysis of a time-spectral algorithm for time
		%  fractional diffusion problems of order $\alpha$ ($0 < \alpha < 1$). The solution regularities 
		%  in Sobolev space and Besov space are established, and it is shown that the temporal regularity in Besov space is twice of that in Sobolev space. Then a time-spectral algorithm is developed, which adopts a standard spectral method and a conforming finite element method for temporal and spatial discretizations, respectively. In addition, optimal error estimates are derived with nonsmooth data. Particularly, sharp temporal convergence rate $1+2\alpha$ is proved theoretically and observed numerically.
	\end{abstract}
	

	\section{Introduction}
	\label{sec:intro}
	This paper considers the following time fractional diffusion problem:
	\begin{equation}
	\label{eq:model}
	\left\{
	\begin{aligned}
	\D_{0+}^\alpha (u-u_0) - \Delta u & = f   &  & \text{in~$ \Omega\times(0,T) $,}         \\
	u                                 & = 0   &  & \text{on $ \partial\Omega\times(0,T) $,} \\
	u(0)                              & = u_0 &  & \text{in $ \Omega $,}
	\end{aligned}
	\right.
	\end{equation}
	where $T>0,\,0 < \alpha < 1 $, $ \D_{0+}^\alpha $ is a Riemann\textendash Liouville fractional
	differential operator (see \cref{sec:pre}), $ \Omega \subset \mathbb R^d $ ($d=1,2,3$) is a convex polygonal domain, and $ u_0 $ and $ f $ are given data. 
	
	Problem \cref{eq:model} is widely used in modeling of  anomalous diffusion process \cite{metzler_fractional_1994,metzler_random_2000} and anomalous transport \cite{luchko_modeling_2011,zaslavsky_chaos_2002},   for its capability of  accurately describing
	models with non-locality and historical memory \cite{kilbas_theory_2006,podlubny_fractional_1999}.  For theoretical study to the problem, e.g. the weak solution and its regularity, we refer to \cite{gorenflo_time-fractional_2015,Li-Xie2019,li_space-time_2009,sakamoto_initial_2011}.   
	
	Many numerical methods have been developed in the past a dozen years. %,      in the recent decade, the state of the art of which will be discussed in the following.
	Among existing works, four types of temporal discretization are most prevailing, i.e.,   finite difference methods (L-type schemes) \cite{alikhanov_new_2015, kopteva_error_2019,lin_finite_2007,lv_error_2016},   convolution quadrature methods \cite{ford_finite_2011,gao_stability_2015,xing_higher_2018,yang_time_2018},  finite element methods  \cite{li_luo_xie_analysis_2019,li_luo_xie_space-time_2019,li_analysis_2019-1,li_numerical_2019} and  spectral  methods \cite{li_luo_xie_time-spectral_2018,li_space-time_2009,shen_efficient_2019,zheng_novel_2015}. 
	Under certain circumstances, problem \cref{eq:model} has an equivalent form like
	\begin{equation}
	\label{eq:model-equi}
	\left\{
	\begin{aligned}
	u_t - \D_{0+}^{1-\alpha}\Delta u & = g   &  & \text{in~$ \Omega\times(0,T) $,}         \\
	u                                 & = 0   &  & \text{on $ \partial\Omega\times(0,T) $,} \\
	u(0)                              & = u_0 &  & \text{in $ \Omega $,}
	\end{aligned}
	\right.
	\end{equation}
	where $g =  \D_{0+}^{1-\alpha} f$. In the literature, both \cref{eq:model,eq:model-equi} are called time fractional diffusion equations or time fractional subdiffusion equations. For  the solution regularity and numerical analysis of problem \cref{eq:model-equi}, especially in the case of nonsmooth data,
	%which is different from that of \cref{eq:model}, 
	we refer the reader to \cite{li_luo_xie_error_2018,McLean2010,mclean_convergence_2009,mustapha_time-stepping_2015,mustapha_uniform_2012,mustapha_superconvergence_2013}. 
	
	It is well-known that the solution to problem \cref{eq:model} generally has boundary singularity (near $0+$) in temporal direction. If $f=0$ and $u_0\neq 0$,
	% is smooth or not, 
	 or $u_0=0$ and $f$ is smooth, then one can obtain   growth estimates of the solution \cite{jin_error_2015,jin_error_2013} or even find out
	the leading singular term $ct^\alpha$ of the solution \cite{Li-Xie2019}. Due to the   singularity, the accuracy order $2-\alpha$ of the L1 scheme \cite{lin_finite_2007} deteriorates into 1 in the   case of $f=0$ and $u_0\neq0$, whether the initial data $u_0$ is smooth or not \cite{jin_analysis_2015}. In the same situation, a piecewise constant discontinuous Galerkin (DG) semidiscretization was analyzed in \cite{McLean2015Time}.  The error estimate results  of \cite{jin_analysis_2015,McLean2015Time} can be   summarized as follows: for any temporal grid node $t_j=j\tau$ with $j=1,2,\cdots, J$ and $\tau= T/J$,
	\begin{equation}\label{eq:linf-est}
	\nm{(u-U)(t_j)}_{L^2(\Omega)}
	\leqslant C
	\left\{
	\begin{aligned}
		&t_j^{\alpha-1}\tau\nm{u_0}_{\dot H^2(\Omega)},&&\text{for L1 in }\cite{jin_analysis_2015}\\
	&t_j^{-1}\tau\nm{u_0}_{L^2(\Omega)},&&\text{for L1 in }\cite{jin_analysis_2015} \text{ and } \text{DG in }\cite{McLean2015Time}.
	\end{aligned}
	\right.
	\end{equation}
	Hence, if $u_0\in L^2(\Omega)$, then the first order accuracy under $L^\infty(0,T;L^2(\Omega))$-norm is only achieved far away from the origin, and  the global convergence rate degenerates as $t_j$ approaches to zero; 
	and if $u_0\in\dot H^2(\Omega)$, then the global rate reduces to $\tau^\alpha$.
	The  estimates in \cref{eq:linf-est} coincide with the 
	solution regularity in Sobolev space (see \cref{thm:regu-u0}):
	\[
\left\{
	\begin{aligned}
{}&\epsilon^{3/2}
\snm{u}_{H^{1/2-\epsilon}(0,T;L^2(\Omega))}
\leqslant {}
C_{\alpha,T} \nm{u_0}_{L^{2}(\Omega)},&&0<\epsilon\leqslant  1/2,\\
{}&	\sqrt{2\gamma-\gamma^2}
	\nm{u}_{H^{(1+\alpha\gamma)/2}(0,T;L^2(\Omega))}
	\leqslant {}
	C_{\alpha,T,\Omega} \nm{u_0}_{\dot H^{\gamma}(\Omega)},&&0<\gamma<2,
\end{aligned}
\right.
	\]
which means that $u_0\in L^2(\Omega)\nRightarrow u\in L^\infty(0,T;L^2(\Omega))$ and that 
\[
u_0\in\dot H^2(\Omega)\Rightarrow u\in H^{(1+\alpha\gamma)/2}(0,T;L^2(\Omega))\hookrightarrow L^\infty(0,T;L^2(\Omega)),
\]
 since $\gamma\in(0,2)$ implies the embedding relation above.
	
	To improve the temporal accuracy,  graded meshes were used in \cite{kopteva_error_2019, mustapha_discontinuous_2014,stynes_error_2017} and some correction techniques were proposed in \cite{ford_approach_2017,jin_correction_2017,li_analysis_2019-2,yan_analysis_2018}. However, most of the existing works using graded meshes require some  assumption of growth estimate on the true solution, and the analyses of  correction schemes for \cref{eq:linf-est} are mainly based on the  Laplace transform, which is only applicable for uniform temporal grids, and the obtained convergence rates have the form $t_j^{-q}\tau^p$ with $0<q\leqslant p$ (like \eqref{eq:linf-est}), which deteriorate near the origin. In \cite{li_numerical_2020}, several technical stability results were developed to establish the optimal first order accuracy of a piecewise constant DG method on   graded meshes.  Spectral methods with singular basis functions were presented  \cite{chen_generalized_2014,shen_efficient_2019}, but so far no rigorous convergence analysis is available with low regularity data. In \cite{duan_exponentially_2019}, a multi-domain Petrov--Galerkin spectral method with a singular basis and geometrically graded meshes was proposed, and the exponential decay was verified numerically with nonsmooth initial data.
	
	In the 1980s, Gui and Babu\v{s}ka \cite{gui_h_1986_1} established the optimal approximation of order $1+2\beta$ under the $L^2$-norm of the Legendre orthogonal expansion for the singular function $(x+1)^\beta$ on $(-1,1)$. Later Babu\v{s}ka and Suri \cite{babuska_optimal_1987} extended this result to a $p$-version finite element method for solving  two dimensional elliptic equations, and proved the sharp accuracy of order $2\beta$ under an energy norm, by assuming that the solution has the explicit singular expression $r^\beta$ around the origin. 
	
	Note that the singular functions mentioned above have boundary singularities as well but the achieved convergence rates   agree with their regularity in the Besov space.
	In view of the boundary singularity of the solution to problem \cref{eq:model}, one may 
	wonder whether this happens to the convergence behavior of a time-spectral method.
	For simplicity, let us start with a fractional ordinary differential equation
	\begin{equation}\label{eq:fode-intro}
	\D_{0+}^\alpha(y-y_0) +  \lambda y = 0\quad \text{in}\quad (0,T),
	\end{equation}
	where $y_0\in\R$ and $\lambda>0$. Invoking the 
	Laplace transform gives the solution expression
\begin{equation}\label{eq:sol-ode-intro}
	y(t) =
y_0\sum_{k=0}^{\infty}
\frac{(-\lambda t^\alpha)^k}{\Gamma(\alpha k+1)},
\quad 0\leqslant t\leqslant T.
\end{equation}
	Note that for a given  fixed (small) $\lambda>0$, we have $y\in H^{1/2+\alpha-\epsilon}(0,T)$ for any $\epsilon>0$ (see \cref{rem:regu-ML}). We adopt a standard Legendre spectral method with polynomial degree $M\in\mathbb N$ to seek an approximation $Y_M$, 
	and use $Y_{50}$ 
	as a reference solution.   \cref{fig:ode_test} plots the convergence order $1+2\alpha$ under  $L^2$-norm  in the case that $\lambda=y_0=T=1$. This agrees with the Besov regularity,  $\bb_{-\alpha,0}^{1+2\alpha-\epsilon}(0,T)$ 	for any $\epsilon>0$, of \cref{eq:sol-ode-intro}; see  \cref{lem:regu-y-g-0-Besov}.
	However, if $\lambda$ is extremely large or goes to infinity, then we can see from \cref{lem:Phi-L2-w} that the convergence 
	rate will be ruined (we also refer the reader to \cite[Section 1.2]{duan_exponentially_2019} for detailed numerical 
	investigations in this case).
	\begin{figure}[H]
		\centering
		\includegraphics[width=252pt,height=216pt]{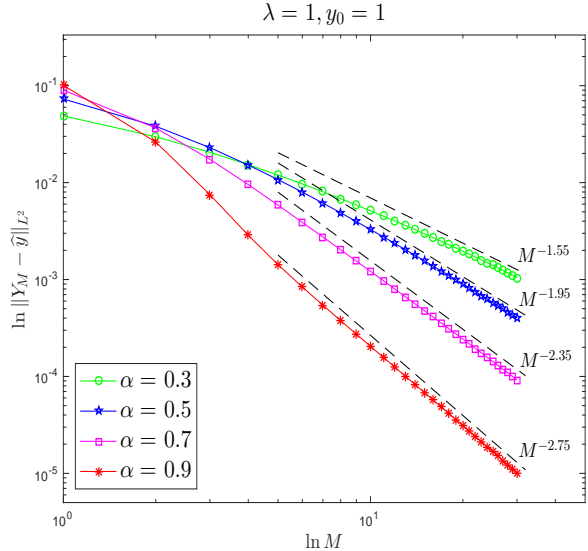}
		\caption{Discretization errors of problem \cref{eq:fode-intro}  with $\lambda= y_0=T=1$. The predicted accuracy is $M^{-1-2\alpha}$.}
		\label{fig:ode_test}
	\end{figure}
	
	As for the model problem \cref{eq:model} itself, although there exists a space-time spectral method   proposed in \cite{li_space-time_2009}, to our best knowledge, no such convergence rate $1+2\alpha$ has been mentioned numerically and established rigorously. 
	In fact, it is nontrivial to obtain this result,  since now the impact of large $\lambda$ comes from the negative Laplacian operator $-\Delta$ (or its discrete version $-\Delta_h$).
%	The deterioration of convergence rate for \cref{eq:fode-intro}  with large $\lambda$  
	This motivates us to revisit the convergence analysis of the time-spectral method for time fractional diffusion problem \cref{eq:model}. Is it possible to prove the optimal approximation order in terms of Besov regularity with nonsmooth data? Especially, whether the accuracy $1+2\alpha$ can be established or improved?  
	
	In this work, we  give   positive answers to these questions mentioned above. Optimal error estimates with respect to the solution regularity in Besov space are established with low regularity data. Moreover, temporal convergence rates $1+\alpha$ and $1+2\alpha$ under $H^{\alpha/2}(0,T;L^2(\Omega))$-norm and $L^2(0,T;\dot H^1(\Omega))$-norm are derived, respectively, which are sharp and cannot be improved even for smoother data. 
	
	The rest of this paper is organized as follows. \cref{sec:pre}  introduces some notations, including standard conventions, functional spaces and fractional calculus operators. 
	 \cref{sec:regu}  defines the weak solution and establishes its regularity results in Sobolev space and Besov space. \cref{sec:main} presents our main error estimates for the time-spectral method, and   \cref{sec:numer} shows several numerical experiments. Finally,   \cref{sec:concl} gives some concluding remarks.
	\section{Preliminary}
	\label{sec:pre}
	For ease of notation, we make some standard conventions.  For a Lebesgue measurable subset $ \omega $ of
	$ \mathbb R^l $ ($l=1,2,3$), we use $ H^\gamma(\omega) $ ($  \gamma \in\R$) and $ H_0^\gamma(\omega) $ ($  \gamma>0$) to denote  two
	standard Sobolev spaces \cite{tartar_introduction_2007}. 
	Given $1\leqslant p<\infty$, if $\omega$ is an interval 
	%on the real line 
	and $\mu$ is a nonnegative measurable function on $\omega $, then $L_\mu^p(\omega)$ denotes the weighted $L^p$-space, and the symbol $ \dual{a,b}_\mu $ means $ \int_{\omega}ab\mu$ whenever $ab\in L^1_{\mu}(\omega)$; if $ \omega $ is a Lebesgue measurable set of $ \mathbb
	R^l ( l= 1,2,3,4 )$, then $ \dual{a,b}_\omega $ stands for $ \int_\omega
	ab $ whenever $ ab\in L^1(\omega) $; if $X$ is a Banach space, then
	$\dual{\cdot,\cdot}_X$ means the duality pairing between $X^*$ (the dual space of $X$) and $ X $. 
	%If $\omega$ is a nonnegative measurable function defined on a Lebesgue measurable set $\omega\subset \mathbb
	%R^l ( l= 1,2,3,4 )$ and the product $pq$ belongs to the weighted space $L^1_{\omega}(\omega)$, then $ \dual{p,q}_\omega $ means $ \int_{\omega}pq \omega$; if $ \omega $ is a Lebesgue measurable set of $ \mathbb
	%R^l ( l= 1,2,3,4 )$, then $ \dual{p,q}_\omega $ stands for $ \int_\omega
	%pq $ whenever $ pq\in L^1(\omega) $; if $X$ is a Banach space, then
	%$\dual{\cdot,\cdot}_X$ means the duality pairing between $X^*$ (the dual space of
	%$X$) and $ X $.
	In particular, if $ X $ is a Hilbert space, then $(\cdot,\cdot)_X$ means its inner product.  If $X$ and $Y$ are two Banach spaces, then $[X,Y]_{\theta,2}$ is the interpolation space constructed by the well-known $K$-method \cite{bergh_interpolation_1976}. For $k\in\mathbb N$ and any $d$-polytope $\omega\subset\R^d(d=1,2,3)$, $P_k(\omega)$ denotes the set of all polynomials defined on $\omega$ with degree no more than $k$.
	
	It is well-known (cf. \cite{evans_partial_nodate}) that   $L^2(\Omega)$ has an orthonormal basis $\{\phi_n \}_{n=0}^\infty $
	%\[
	%\{\phi_n: i \in \mathbb N \} \subset H_0^1(\Omega) \cap H^2(\Omega)
	%\]
	%of $ L^2(\Omega) $ 
	such that 
	\[
	\left\{
	\begin{aligned}
	-\Delta \phi_n ={}& \lambda_n\phi_n,&&\text{in~}\Omega,\\
	\phi_n={}&0,&&\text{on~}\partial \Omega,
	\end{aligned}
	\right.
	\]
	where $ \{ \lambda_n\}_{n=0}^\infty$ is a
	nondecreasing real positive sequence and $\lambda_0 = \nm{\nabla\phi_0}^2_{L^2(\Omega)}>0$ depends only on $\Omega$. 
	For any $  \gamma \in\R$, define
	\[
	\dot H^\gamma(\Omega) := \left\{
	\sum_{n=0}^\infty c_n \phi_n:\
	\sum_{n=0}^\infty  \lambda_n^\gamma c_n^2 < \infty
	\right\}
	\]
	and equip this space with the inner product  %such that
	\[
	\left(
	\sum_{n=0}^\infty c_n \phi_n,
	\sum_{n=0}^\infty d_n\phi_n
	\right)_{\dot H^\gamma(\Omega)} :=
	\sum_{n=0}^\infty \lambda_n^\gamma c_n d_n, \quad \text{for all }\sum_{n=0}^\infty c_n \phi_n, \sum_{n=0}^\infty d_n\phi_n\in \dot
	H^\gamma(\Omega) .
	\]
	%for all $ \sum_{n=0}^\infty c_n \phi_n, \sum_{n=0}^\infty d_n\phi_n\in \dot
	%H^\gamma(\Omega) $.
	The induced norm is denoted by $\nm{\cdot}_{\dot H^\gamma(\Omega)} = \sqrt{\inner{\cdot,\cdot}_{\dot H^\gamma(\Omega)}}$. 
	Note that $\dot H^\gamma(\Omega)$ is a Hilbert space and has an 
	orthonormal basis $\{\lambda_n^{-\gamma/2}\phi_n\}_{n=0}^{\infty}$.
	
	%\textit{\textbf{Sobolev Spaces.}} Let $ X $ be a separable Hilbert
	%space with an inner product $ (\cdot,\cdot)_X $ and an orthonormal basis $
	%\{e_i: i \in \mathbb N\} $. Assuming that $-\infty<a<b<\infty$, we use $ H^\gamma(a,b;X) $ ($ 0 \leqslant \gamma < 
	%\infty $) to denote a usual vector valued Sobolev space, endowed with the norm
	%\begin{equation}\label{eq:def-HX}
	%  \nm{v}_{H^\gamma(a,b;X)} := \left(
	%\sum_{n=0}^\infty \nm{(v,e_i)_X}_{H^\gamma(0,T)}^2
	%\right)^{1/2} \quad \forall v \in H^\gamma(a,b;X).
	%\end{equation}
	%For $ 0 < \gamma < 1/2 $, we also use the norm
	%\[
	%  \snm{v}_{H^\gamma(a,b;X)} := \left(
	%    \sum_{n=0}^\infty \snm{(v,e_i)_X}_{H^\gamma(0,T)}^2
	%  \right)^{1/2} \quad \forall v \in H^\gamma(a,b;X).
	%\]
	%Here the norm $ \snm{\cdot}_{H^\gamma(0,T)} $ is given by
	%\[
	%  \snm{w}_{H^\gamma(0,T)} := \left(
	%    \int_\mathbb R \snm{\xi}^{2\gamma}
	%    \snm{\mathcal F(w\chi_{(0,T)})(\xi)}^2 \,\mathrm{d}\xi
	%  \right)^{1/2} \quad \forall w \in H^\gamma(0,T),
	%\]
	%where $ \mathcal F: L^2(\mathbb R) \to L^2(\mathbb R) $ is the Fourier transform
	%operator and $ \chi_{(0,T)} $ is the indicator function of $ (0,T) $.
	
	Given any $0<\gamma<2$, we introduce the space 
	\[
	{}_0H^{\gamma}(0,T):=[L^2(0,T),{}_0H^{2}(0,T)]_{\gamma/2,2},
	\]
	where ${}_0H^{2}(0,T): = \left\{
	v\in H^2(0,T):v(0) = v'(0) = 0
	\right\}$ with  norm  $\nm{v}_{{}_0H^{2}(0,T)} := \nm{v''}_{L^2(0,T)}$. Therefore, using the interpolation 
	theorem of bounded linear operators \cite[Theorem 1.6]{lunardi_interpolation_1995} yields  
	\begin{equation}\label{eq:0H>H}
	\nm{v}_{[L^2(0,T),H^{2}(0,T)]_{\gamma/2,2}}
	\leqslant \nm{v}_{{}_0H^{\gamma}(0,T)}
	\quad\forall\,v\in{}_0H^{\gamma}(0,T).
	\end{equation}
	In addition, if $0<\gamma<1/2$, then by \cite[Chapter 1]{lions_non-homogeneous3_1973}, the relation ${}_0H^{\gamma}(0,T)=H^{\gamma}(0,T)$ holds in the sense of equivalent norms, and in this case (i.e., $0<\gamma<1/2$) we have an alternative norm, which is defined by
	\[
	\snm{w}_{H^\gamma(0,T)} := \left(
	\int_\mathbb R \snm{\xi}^{2\gamma}
	\snm{\mathcal F(w\chi_{(0,T)})(\xi)}^2 \,\mathrm{d}\xi
	\right)^{1/2} \quad \forall w \in H^\gamma(0,T),
	\]
	where $ \mathcal F: L^2(\mathbb R) \to L^2(\mathbb R) $ is the Fourier transform and $ \chi_{(0,T)} $ is the indicator function of $ (0,T) $. 
	
	Let $ X $ be a separable Hilbert
	space with an inner product $ (\cdot,\cdot)_X $ and an orthonormal basis $
	\{e_n: n \in \mathbb N\} $.
	For any $\gamma\in\R$, let $ H^\gamma(0,T;X) $ be a usual vector-valued Sobolev space defined by
	\begin{equation}\label{eq:def-HX}
	H^\gamma(0,T;X) := \left\{\sum_{n=0}^{\infty}v_ne_n
	:\,\sum_{n=0}^\infty \nm{v_n}_{H^\gamma(0,T)}^2<\infty
	\right\},
	\end{equation}
%	endowed 
	with the norm
	\begin{equation*}%\label{eq:def-HX}
		\nm{v}_{H^\gamma(0,T;X)} := \left(
		\sum_{n=0}^\infty \nm{(v,e_n)_X}_{H^\gamma(0,T)}^2
		\right)^{1/2} \quad \forall\,v \in H^\gamma(0,T;X).
	\end{equation*}
	The space ${}_0H^{\gamma}(0,T;X)$ for $0<\gamma<2$ is defined in a similar way as \cref{eq:def-HX}.
	
	For $a,b>-1$, let $\{S_k^{a,b}\}_{k=0}^\infty$ be the family of shifted Jacobi polynomials on $(0,T)$ with respect to the weight $\mu^{a,b}(t) = (T-t)^at^b$; 
	see \cref{sec:sJacobi-append}. 
	%It is clear that $\{S_k^{a,b}\}_{k=0}^\infty$ forms 
	%a complete orthogonal basis of $L^2_{\mu^{a,b}}(0,T)$. Hence, for any $v\in L^2_{\mu^{a,b}}(0,T)$, we have a unique orthogonal decomposition
	%\[
	%v=\sum_{k=0}^{\infty}
	%v_kS_k^{a,b},\quad v_k = \frac{1}{\xi_k^{a,b}}
	%\dual{v,S_{k}^{a, b}}_{\mu^{a,b}},\quad
	%\xi_k^{a,b} = \nm{S_k^{a,b}}_{L^2_{\mu^{a,b}}(0,T)}^2,
	%\]
	%and an equivalent definition of $L^2_{\mu^{a,b}}(0,T)$ is
	%\[
	% L^2_{\mu^{a,b}}(0,T) =\left\{
	% \sum_{k=0}^{\infty}
	% v_kS_k^{a,b}:\,
	% \sum_{k=0}^{\infty}\xi_k^{a,b}v_k^2<\infty,\,
	%\xi_k^{a,b} = \nm{S_k^{a,b}}_{L^2_{\mu^{a,b}}(0,T)}^2
	% \right\}.
	%\]
	Given $\gamma\geqslant 0$, we introduce the Besov space (also known as the weighted Sobolev space, cf. \cite{babuska_direct_2002}) defined by 
	\begin{equation}\label{eq:Habg}
	\bb_{a,b}^\gamma(0,T): = \left\{
	\sum_{k=0}^{\infty}
	v_kS_k^{a,b}:\,
	\sum_{k = 0}^{\infty} (1+k^{2\gamma})\xi_k^{a,b}v_k^2<\infty
	%\xi_k^{a,b} = \nm{S_k^{a,b}}_{L^2_{\mu^{a,b}}(0,T)}^2
	\right\},
	\end{equation}
	where $\xi_k^{a,b}$ is given by \cref{eq:xiab}, and endow this space with the norm
	\[
	\nm{v}_{\bb_{a,b}^\gamma(0,T)}: = 
	\left(\sum_{k = 0}^{\infty}
	(1+k^{2\gamma})
	\xi_k^{a,b}v_k^2\right)^{1/2}
	\quad \forall\,
	v= \sum_{k=0}^{\infty}
	v_kS_k^{a,b}\in \bb_{a,b}^\gamma(0,T).
	\]
	In addition, for any separable Hilbert space $X$, the vector-valued space $\bb_{a,b}^\gamma(0,T;X)$ can be defined in a similar way as that of \cref{eq:def-HX}.
	
To the end, let us  introduce the Riemann--Liouville fractional calculus operators and list some important lemmas. For any $\gamma>0$ and $v\in L^1(0,T;X)$, define the fractional integrals of order $\gamma$ as follows:
	\begin{align*}
		\left(\D_{0+}^{-\gamma} v\right)(t) &:=
		\frac1{ \Gamma(\gamma) }
		\int_0^t (t-s)^{\gamma-1} v(s) \, \mathrm{d}s, 
		\quad t\in(0,T), \\
		\left(\D_{T-}^{-\gamma} v\right)(t) &:=
		\frac1{ \Gamma(\gamma) }
		\int_t^T (s-t)^{\gamma-1} v(s) \, \mathrm{d}s, 
		\quad t\in(0,T), 
	\end{align*}
	where $ \Gamma(\cdot) $ denotes the Gamma function
	\begin{equation}\label{eq:Gamma}
	\Gamma(z): = \int_{0}^{\infty}t^{z-1}e^{-t}\dd t,\quad z>0.
	\end{equation}
	For $ k < \gamma < k+1 $ with $ k \in \mathbb N $, define the left-sided and right-sided Riemann--Liouville  fractional derivative operators of order  $\gamma $  respectively  by 
	\[
	\D_{0+}^\gamma  := {}\D^k \D_{0+}^{\gamma-k},\quad
	\D_{T-}^\gamma  :={} (-\D)^k \D_{T-}^{\gamma-k},
	\]
	where $ \D $ is the first-order generalized derivative operator. 
	\begin{lem}[\cite{ervin_variational_2006}]
		\label{lem:coer}
		If $ -1/2< \gamma < 1/2 $ and $ v,\,w \in
		H^{\max\{0,\gamma\}}(0,T) $, then
		\begin{align*}
			&				\dual{\D_{0+}^\gamma v, \D_{T-}^\gamma v}_{(0,T)}
			={}\cos(\gamma\pi) \snm{v}_{H^{\gamma}(0,T)}^2,\\
			&		\cos(\gamma\pi)\nm{\D_{0+}^\gamma v}_{L^2(0,T)}^2\leqslant
			\dual{\D_{0+}^\gamma v, \D_{T-}^\gamma v}_{(0,T)}
			\leqslant {}\sec(\gamma\pi) \nm{\D_{0+}^\gamma v}_{L^2(0,T)}^2,\\
			&		 \dual{\D_{0+}^{2\gamma} v, w}_{H^\gamma(0,T)} = \dual{\D_{0+}^\gamma v, \D_{T-}^\gamma w}_{(0,T)} \leqslant{} \cos(\gamma\pi)
			\snm{v}_{H^\gamma(0,T)} \snm{w}_{H^\gamma(0,T)}.
		\end{align*}
	\end{lem}
	\begin{lem}[\cite{luo_li_xie_convergence_2019}]
		\label{lem:key-equi}
		If $ v\in{}_0 H^{\gamma}(0,T) $ with $0< \gamma<2$, then
		\[
		C_1\nm{\D_{0+}^{\gamma}v}_{L^2(0,T)}
		\leqslant \nm{v}_{{}_0H^{\gamma}(0,T)}
		\leqslant C_2\nm{\D_{0+}^{\gamma}v}_{L^2(0,T)},
		\]
		where $C_1$ and $C_2$ depend only on $\gamma$.
	\end{lem}
	\section{Weak Solution and Regularity}
	\label{sec:regu}
	This section is   to revisit the solution regularity of problem \cref{eq:model} 
	in terms of proper Sobolev spaces and establish new regularity results in Besov spaces. 
	
	Following \cite{li_luo_xie_analysis_2019,li_space-time_2009}, we first introduce the 
	weak solution to problem \cref{eq:model}. To do so, set
	\begin{equation}\label{X-def}
	%\mathcal G :={}& {}_0H^{\alpha/2}(0,T;\dot H^1(\Omega)),\\
	\mathcal X :={} H^{\alpha/2}(0,T;L^{2}(\Omega)) \cap L^{2}(0,T;\dot H^1(\Omega)),
	\end{equation}
	and endow this space with the norm
	\[
	\begin{split}
	\nm{\cdot}_{\mathcal X} :={}&
	\left(  \snm{\cdot}_{H^{\alpha/2}(0,T;L^{2}(\Omega))}^2 +
	\nm{\cdot}_{L^{2}(0,T;\dot H^1(\Omega))}^2\right)^{1/2}.
	\end{split}
	\]
	Assuming that $f+\D_{0+}^\alpha u_0  
	\in \mathcal X^*$, we call $ u \in \mathcal X $ 
	a weak solution to problem \cref{eq:model} if
	\begin{equation}
	\label{eq:weak_form}
	\dual{ \D_{0+}^\alpha u, v }_{ H^{\alpha/2}( 0,T;\dot H^1(\Omega) ) } +
	\dual{ \nabla u, \nabla v }_{ H^{\alpha/2}( 0,T;L^2(\Omega) ) } =
	\dual{f+\D_{0+}^\alpha u_0,v}_{\mathcal X}\quad \forall\, v\in \mathcal X.
	\end{equation}
	
	As mentioned in \cite[Remark 2.2]{li_luo_xie_analysis_2019}, the well-posedness of the weak formulation \cref{eq:weak_form} follows from the Lax--Milgram theorem and \cref{lem:coer}. More precisely, if $f+\D_{0+}^\alpha u_0  
	\in \mathcal X^*$, then problem \cref{eq:model} admits 
	a unique weak solution in the sense of \cref{eq:weak_form} such that
	\[
	\nm{u}_{\mathcal X} \leqslant 
	C_{\alpha}
	\nm{f+\D_{0+}^\alpha u_0 }_{\mathcal X^*}.
	\]
	
	To establish   more elaborate regularity estimates, we apply the Galerkin method that reduces \cref{eq:weak_form} to a family of ordinary differential equations, to 
	which the solutions  can be used to recover the weak solution to \cref{eq:weak_form} through a series expression; see the lemma below.
	\begin{prop}\label{lem:decomp-u}
		Assume $f\in L^2(0,T;\dot H^{-1}(\Omega))$ and $u_0\in\dot H^{\gamma}(\Omega)$ where $\gamma=-1$ if $0<\alpha<1/2$, and $\gamma>1-1/\alpha$ if $1/2\leqslant \alpha<1$. The solution to \cref{eq:weak_form} is given by $u = \sum_{n=0}^{\infty}y_n\phi_n$, where $y_n\in H^{\alpha/2}(0,T)$ satisfies
		\begin{equation}\label{eq:weak-form-yn}
		\left\langle\mathrm{D}_{0+}^{\alpha}\left(y_n-y_{n,0}\right), z\right\rangle_{H^{\alpha / 2}(0, T)}+\lambda_n\langle y_n, z\rangle_{(0, T)}=\langle f_n, z\rangle_{(0, T)},
		\end{equation}
		for all $z\in H^{\alpha/2}(0,T)$, where $y_{n,0} = \dual{u_0,\phi_n}_{\dot H^{\gamma}(\Omega)}$ and $f_n = \dual{f,\phi_n}_{\dot H^{1}(\Omega)}$.
	\end{prop}	 
	\begin{proof}
		The proof here is actually in line with that of \cite[Theorem 3.1]{li_luo_xie_analysis_2019}, 
		where the case $0<\alpha<1/2$ has been considered. 
		The case of $1/2\leqslant \alpha<1$ follows similarly.
		%	 	
		%	 	We first note that by definition $y_{n,0}\in \R$ and $f_n\in L^2(0,T)$ are well defined and it holds that
		%	 	\[
		%\begin{aligned}
		%	 	u_0 ={}& \sum_{n=0}^{\infty}\lambda_n^{\gamma}y_{n,0}\phi_n,&&\nm{u_0}^2_{\dot H^{-\gamma}(\Omega)} = \sum_{n=0}^{\infty}\lambda_n^{\gamma}\snm{y_{n,0}}^2,\\
		%f ={}& \sum_{n=0}^{\infty}\lambda_nf_n\phi_n,&&\nm{f}_{L^2(0,T;\dot H^{-1}(\Omega))}^2 = \sum_{n=0}^{\infty}\lambda_n^{-1}\nm{f_n}^2_{L^2(0,T)}.
		%\end{aligned}
		%	 	\]
		%	 	 Similar with \cref{eq:weak_form}, we conclude from the Lax--Milgram theorem and \cref{lem:coer} that $y_n\in H^{\alpha/2}(0,T)$ is uniquely defined by \cref{eq:weak-form-yn} and satisfies 
		%	 	 \[
		%	 	 \nm{y_n}_{H^{\alpha/2}(0,T)}+\lambda_n^{1/2}\nm{y_n}_{L^2(0,T)}
		%	 	 \leqslant C_{\alpha}\nm{f_n}_{L^2(0,T)}+C_{\alpha,T}\snm{y_{n,0}}.
		%	 	 \]
		%	 	 It follows that
		%	 	 \[
		%	 	 \sum_{n=0}^{\infty}\left(\nm{y_n}^2_{H^{\alpha/2}(0,T)}+\lambda_n\nm{y_n}^2_{L^2(0,T)}\right)
		%	 	 \]
	\end{proof}
	%\begin{rem}
	%	 	It is evident that \cref{eq:weak-form-yn} implies
	%	\begin{equation}\label{eq:ode-yn}
	%	\mathrm{D}_{0+}^{\alpha}\left(y_n-y_{n,0}\right)+\lambda_n y_n = f_n\quad{\rm in~}L^2(0,T).
	%	\end{equation}
	%%	In what follows, we consider two circumstances: homogeneous case $f=0$ and $f$
	%	\end{rem}
	\subsection{Regularity in Sobolev space}
	We first revisit the Sobolev regularity of the solution to \cref{eq:weak_form}. Thanks to \cref{lem:decomp-u}, this can be done by investigating problem \cref{eq:weak-form-yn}, which, in a general form, is equivalent to
	\begin{equation}\label{eq:ode-y-f-0}
	\D_{0+}^{\alpha}(y-y_0)+\lambda y=g,
	\end{equation}
	where $\lambda>0,\,y_0\in \R$ and $g\in L^2(0,T)$. In fact, in \cite[Lemmas 3.1 and 3.2]{li_luo_xie_analysis_2019} we have established corresponding regularity results via a variational approach:
	\begin{lem}[\cite{li_luo_xie_analysis_2019}]
		If $y_0=0$, then the unique solution $y$ to \cref{eq:ode-y-f-0} satisfies
		\[
		\nm{y}_{{}_0H^{\alpha}(0,T)}
		+\lambda^{1/2}\nm{y}_{H^{\alpha/2}(0,T)}
		+\lambda\nm{y}_{L^2(0,T)}\leqslant C_{\alpha,T}\nm{g}_{L^2(0,T)};
		\]
		and if $g=0$, then 
		\begin{equation}\label{eq:regu-y0}
		\nm{y}_{H^{\alpha}(0,T)}
		+\lambda^{1/2}\nm{y}_{H^{\alpha/2}(0,T)}
		+\lambda\nm{y}_{L^2(0,T)}\leqslant C_{\gamma,\alpha,T}\lambda^{\gamma/2}\snm{y_0},
		\end{equation}
		where $\gamma=0$ if $0<\alpha<1/2$, and 
		$2-1/\alpha<\gamma\leqslant 1$ if $1/2\leqslant \alpha<1$.
	\end{lem}
	This lemma, together with \cref{lem:decomp-u}, implies the following results (see \cite[Theorems 3.1 and 3.2]{li_luo_xie_analysis_2019}):
	\begin{thm}[\cite{li_luo_xie_analysis_2019}]
		If $u_0=0$ and $f\in L^2(0,T;\dot H^{\gamma}(\Omega))$ with $-1\leqslant \gamma\leqslant 0$, then the weak solution defined by \cref{eq:weak_form} satisfies
		\[
		\nm{u}_{{}_0H^{\alpha(1+\gamma)/2}(0,T;L^{2}(\Omega))}
		+\nm{u}_{H^{\alpha/2}(0,T;\dot H^{1+\gamma}(\Omega))}
		+ \nm{u}_{L^2(0,T;\dot H^{2+\gamma}(\Omega))}
		\leqslant C_{\alpha,T}\nm{f}_{L^2(0,T;\dot H^{\gamma}(\Omega))}.
		\]
		If $f=0,\,0<\alpha<1/2$ and $u_0\in \dot H^{\gamma}(\Omega)$ with $-1\leqslant \gamma\leqslant 0$, then
		\begin{equation}\label{eq:regu-u0-1}
		\nm{u}_{H^{\alpha(1+\gamma)/2}(0,T;L^{2}(\Omega))}
		+\nm{u}_{H^{\alpha/2}(0,T;\dot H^{1+\gamma}(\Omega))}
		+ \nm{u}_{L^2(0,T;\dot H^{2+\gamma}(\Omega))}
		\leqslant C_{\alpha,T}\nm{u_0}_{\dot H^{\gamma}(\Omega)}.
		\end{equation}
		Besides, if $f=0,\,1/2\leqslant \alpha<1$ and $u_0\in L^{2}(\Omega)$, then
		\begin{equation}\label{eq:regu-u0-2}
		\nm{u}_{H^{\alpha(1+\gamma)/2}(0,T;L^{2}(\Omega))}
		+\nm{u}_{H^{\alpha/2}(0,T;\dot H^{1+\gamma}(\Omega))}
		+ \nm{u}_{L^2(0,T;\dot H^{2+\gamma}(\Omega))}
		\leqslant C_{\gamma,\alpha,T}\nm{u_0}_{L^{2}(\Omega)},
		\end{equation}
		where $2-1/\alpha<\gamma\leqslant 1$.
	\end{thm}	
	
	However, we mention that the implicit constants in \cref{eq:regu-y0,eq:regu-u0-2} will blow up when $\gamma\to2-1/\alpha$ and that both \cref{eq:regu-y0,eq:regu-u0-1} are not optimal. 
	Therefore, in this section, we mainly focus on improving \cref{eq:regu-y0,eq:regu-u0-1} and finding explicit relation with respect to the constant $\gamma$, by using the Mittag-Leffler function \cite{mittag_1903}
	\begin{equation}
	\label{eq:E-ab}
	E_{\alpha,\beta}(z) := \sum_{k=0}^\infty
	\frac{z^k}{\Gamma(\alpha k + \beta)}, 
	\quad z \in \mathbb C,\quad \beta\in\R.
	\end{equation}
	Given any $ t > 0 $, it is well-known that (cf. \cite{kilbas_theory_2006})
	\begin{align}
		\snm{E_{\alpha,\beta}(-t)} \leqslant{}&
		\frac{C_{\alpha,\beta}}{1+t}.
		\label{eq:Eabz-est} 
	\end{align}
	In addition, by using the Laplace transform, it is not hard to find the 
	solution to \cref{eq:ode-y-f-0} with $g=0$:
	\begin{equation}\label{eq:ode-y0-y}
	y(t)=y_0E_{\alpha,1}(-\lambda t^\alpha),\quad 0\leqslant t\leqslant T.
	\end{equation}
	\begin{lem}
		\label{lem:regu-y0}
		Assume $\lambda\geqslant \lambda_*>0$, then the function $y(t)$ 
		defined by \cref{eq:ode-y0-y} satisfies 
		\begin{equation}\label{eq:regu1-y-f-0-<1/2}
		\eta(\epsilon)\lambda^{\beta/2-1-\epsilon} \nm{y}_{H^{\alpha}(0,T)} +
		\lambda^{(\beta-1)/2}\nm{y}_{H^{\alpha/2}(0,T)} +
		\eta(\epsilon)\lambda^{\beta/2-\epsilon} \nm{y}_{L^2(0,T)} 
		\leqslant {}C_{\alpha,\lambda_*,T}\snm{y_0},
		\end{equation}
		where $\beta=\min\{2,1/\alpha\}$ and
		\begin{equation}
		\label{eq:eta-eps}
		\left\{
		\begin{aligned}
		&\eta(\epsilon ) = 1,&&\epsilon = 0,&&\text{ if }\alpha\neq 1/2,\\
		&\eta(\epsilon) = \sqrt{\epsilon},&&\epsilon\in(0,1/2],&&\text{ if }\alpha= 1/2.
		\end{aligned}
		\right.
		\end{equation}
		Moreover, we have
\begin{numcases}{}
\label{eq:y0-L2}
\epsilon^{3/2}
\snm{y}_{H^{1/2-\epsilon}(0,T)}
\leqslant {}C_{\alpha,T}	
\snm{y_0},\quad 0<\epsilon\leqslant 1/2,\\
\label{eq:y0-Hs}
\sqrt{2\gamma-\gamma^2}
\nm{y}_{H^{(1+\alpha\gamma)/2}(0,T)}
\leqslant {}C_{\alpha,\lambda_*}	
\lambda^{\gamma/2}\snm{y_0},\quad 0<\gamma<2.
\end{numcases}
%\begin{align}
%	\label{eq:regu2-y-f-0}
%&\epsilon
%\snm{y}_{H^{1/2-\epsilon}(0,T)}
%\leqslant {}C_{\alpha,T}	
%\snm{y_0},&&0<\epsilon<1/2,\\
%	\label{eq:regu2-y-f-0-<1/2}
%&\sqrt{2\gamma-\gamma^2}
%\nm{y}_{H^{(1+\alpha\gamma)/2}(0,T)}
%\leqslant {}C_{\alpha,\lambda_*}	
%\lambda^{\gamma/2}\snm{y_0},&&0<\gamma<2.
%\end{align}
	\end{lem}
	\begin{proof}
		We first prove \cref{eq:regu1-y-f-0-<1/2}. Since the case $0<\alpha<1/2$ has 
		been given by \cref{eq:regu-y0}, we only consider the case $1/2\leqslant \alpha<1$, which says that 
		\begin{equation}\label{eq:1/2-1}
		\eta(\epsilon)\lambda^{-1-\epsilon} \nm{y}_{H^{\alpha}(0,T)} +
		\lambda^{-1/2}\nm{y}_{H^{\alpha/2}(0,T)} +
		\eta(\epsilon)\lambda^{-\epsilon} \nm{y}_{L^2(0,T)} 
		\leqslant {}C_{\alpha,\lambda_*,T}	\lambda^{-\frac{1}{2\alpha}}\snm{y_0}.
		\end{equation}
		By \cref{eq:Eabz-est} and direct calculations, we get
		\begin{equation}\label{eq:mid}
		\begin{split}
		{}&
		\snm{\eta(\epsilon)}^2
		\lambda^{-2-2\epsilon}
		\snm{\left(\D_{0+}^{\alpha}(y-y_0)\right)(t)}^2
		+
		\lambda^{-1}
		\snm{\left(\D_{0+}^{\alpha/2}y\right)(t)}^2
		+
		\snm{\eta(\epsilon)}^2
		\lambda^{-2\epsilon} \snm{y(t) }^2\\
		\leqslant{} &C_\alpha \snm{y_0}^2
		\frac{\snm{\eta(\epsilon)}^2
			\lambda^{-2\epsilon}+\lambda^{-1}
			t^{-\alpha}}{(1+\lambda t^{\alpha})^2} ,
		\end{split}
		\end{equation}
		for all $0<t\leqslant T$.
		It is evident that
		\begin{equation}\label{eq:est-intg}
		\int_{0}^{T}
		\frac{	\lambda^{-1}t^{-\alpha}}{(1+\lambda t^{\alpha})^2}\dd t
		\leqslant \int_{0}^{\infty}
		\frac{	\lambda^{-1}t^{-\alpha}}{(1+\lambda t^{\alpha})^2}\dd t
		={}\int_{0}^{\lambda^{-\frac{1}{\alpha}}}\lambda^{-1}t^{-\alpha}\dd t
		+ \int_{\lambda^{-\frac{1}{\alpha}}}^\infty
		\lambda^{-3}t^{-3\alpha}\dd t
		={}C_\alpha\lambda^{-1/\alpha}.
		\end{equation}
		If $\alpha=1/2$, then for $0<\epsilon\leqslant 1/2$ it holds
		\[
		\begin{split}
		\int_{0}^{T}\frac{	\lambda^{-2\epsilon}}{(1+\lambda t^{\alpha})^2}\dd t
		={}&
		%\lambda^{-\frac{1}{\alpha}}\int_{0}^T
		%\frac{	\lambda^{1/\alpha-2\epsilon}}{(1+\lambda t^{\alpha})^2}\dd t
		%=
		\lambda^{-\frac{1}{\alpha}}
		\int_{0}^T
		\frac{	(\lambda t^\alpha)^{1/\alpha-2\epsilon}}{(1+\lambda t^{\alpha})^2}
		t^{2\alpha\epsilon-1}\dd t
		\leqslant{}
		\lambda^{-\frac{1}{\alpha}}
		\int_{0}^T	t^{2\alpha\epsilon-1}\dd t = 
		\frac{C_{\alpha,T}}{\epsilon}
		\lambda^{-1/\alpha};
		\end{split}
		\]
		and if $1/2<\alpha<1$, then using a similar manner 
		for estimating \cref{eq:est-intg} gives
		\[
		\int_{0}^{T}	\frac{1}{(1+\lambda t^{\alpha})^2}\dd t
		\leqslant 
		%	\int_{0}^{\infty}
		%	\frac{1}{(1+\lambda t^{\alpha})^2}\dd t
		%={}\int_{0}^{\lambda^{-\frac{1}{\alpha}}} \dd t
		%+ \int_{\lambda^{-\frac{1}{\alpha}}}^\infty
		%\lambda^{-2}t^{-2\alpha}\dd t
		%={}
		C_\alpha\lambda^{-1/\alpha}.
		\]
		Hence, by \cref{lem:key-equi}, plugging the above estimates into \cref{eq:mid} implies
		\[
		\eta(\epsilon)\lambda^{-1-\epsilon} \nm{y-y_0}_{{}_0H^{\alpha}(0,T)} +
		\lambda^{-1/2}\nm{y}_{{}_0H^{\alpha/2}(0,T)} +
		\eta(\epsilon)\lambda^{-\epsilon} \nm{y}_{L^2(0,T)} 
		\leqslant {}C_{\alpha,T}	\lambda^{-\frac{1}{2\alpha}}\snm{y_0},
		\]
		which, together with the fact \cref{eq:0H>H} and the assumption 
		$\lambda\geqslant \lambda_*>0$, yields \cref{eq:1/2-1} immediately.

If $0<\epsilon\leqslant 1/2$, then using \cref{lem:coer,eq:Eabz-est} gives
\[
\begin{split}
\snm{y}^2_{H^{1/2-\epsilon}(0,T)} \leqslant \csc^2\epsilon\pi\nm{\D_{0+}^{1/2-\epsilon}y}_{L^2(0,T)}^2
\leqslant {}&
\frac{C_{\alpha}	 \snm{y_0}^2}{\epsilon^2}
\int_{0}^{T}
\frac{t^{2\epsilon-1}}{(1+\lambda t^{\alpha})^2}
\dd t\\
\leqslant {}&
\frac{C_{\alpha}	 \snm{y_0}^2}{\epsilon^2}
\int_{0}^{T}
t^{2\epsilon-1}
\dd t = \frac{C_{\alpha,T}	 \snm{y_0}^2}{\epsilon^3}.
\end{split}
\]
which proves \eqref{eq:y0-L2}.

To the end, we consider \eqref{eq:y0-Hs}. 
		Since $0<\gamma<2$, we have 
		\[
		\frac{1}{2}<\frac{1+\alpha\gamma}{2}<\alpha+1/2.
		\]
		%	If $-2<\gamma<0$, then 
		%	by \cref{eq:Eabz-est,lem:coer} and straightforward computations, 
		%	\[
		%	\begin{split}
		%	\snm{y}_{H^{(1+\alpha\gamma)/2}(0,T)}^2
		%	\leqslant {}&\frac{C_{\alpha}		}{\snm{\gamma}^2}
		%	\int_{0}^{T}		\snm{\left(\D_{0+}^{(1+\alpha\gamma)/2}y\right)(t)}^2\dd t
		%	\leqslant {}\frac{C_{\alpha}	 \snm{y_0}^2	}{\snm{\gamma}^2}	\int_{0}^{T}
		%	\frac{t^{ -\alpha\gamma-1}}{(1+\lambda t^{\alpha})^2}
		%	\dd t=\frac{C_{\alpha}\lambda^{\gamma}	 \snm{y_0}^2	}{\gamma^3(2+\gamma)},
		%	\end{split}
		%	\]
		%which yields that
		%	\[
		%	\begin{split}
		%	\snm{y}_{H^{(1+\alpha\gamma)/2}(0,T)}
		%	\leqslant {}\frac{C_{\alpha}\lambda^{\gamma/2}\snm{y_0}}{\snm{\gamma}\sqrt{\gamma^2-2\gamma}}.
		%	\end{split}
		%	\]	
		%If $0<\gamma<2$, then by \cref{eq:Eabz-est,lem:key-equi},
		By \cref{eq:Eabz-est}, \cref{lem:key-equi} and straightforward calculations, we get
		\[
		\begin{split}
		\nm{y-y_0}_{{}_0H^{(1+\alpha\gamma)/2}(0,T)}^2
		\leqslant {}&C_{\alpha}		
		\int_{0}^{T}		\snm{\left(\D_{0+}^{(1+\alpha\gamma)/2}(y-y_0)\right)(t)}^2\dd t
		\leqslant {}C_{\alpha}		 \snm{y_0}^2\int_{0}^{T}
		\frac{\lambda^{2}t^{ \alpha(2-\gamma)-1}}{(1+\lambda t^{\alpha})^2}
		\dd t,
		\end{split}
		\]
		and we estimate the integral in a similar way for \cref{eq:est-intg} to obtain
		%\[
		%\int_{0}^{T}
		%\frac{\lambda^{2}t^{ \alpha(2-\gamma)-1}}{(1+\lambda t^{\alpha})^2}
		%\dd t
		%\leqslant \int_{0}^{\infty}
		%\frac{\lambda^{2}t^{ \alpha(2-\gamma)-1}}{(1+\lambda t^{\alpha})^2}\dd t
		%={}\int_{0}^{\lambda^{-\frac{1}{\alpha}}}
		%\lambda^{2}t^{ \alpha(2-\gamma)-1}\dd t
		%+ \int_{\lambda^{-\frac{1}{\alpha}}}^\infty
		%t^{ -\alpha\gamma-1}\dd t
		%={}\frac{C_\alpha\lambda^{\gamma}}{2\gamma-\gamma^2}.
		%\]
		\[
		\begin{split}
		\nm{y-y_0}_{{}_0H^{(1+\alpha\gamma)/2}(0,T)}
		\leqslant {}\frac{C_{\alpha}\lambda^{\gamma/2}\snm{y_0}}
		{\sqrt{2\gamma-\gamma^2}}.
		\end{split}
		\]
		Therefore, from   \cref{eq:0H>H} and 
		the assumption $\lambda\geqslant \lambda_*>0$ it follows  \eqref{eq:y0-Hs}.
	\end{proof}
	\begin{rem}
		\label{rem:regu-ML}
		For any fixed $\lambda>0$, from \cref{lem:regu-y0} we see 
		the highest regularity of $y = y_0 E_{\alpha,1}(-\lambda t^\alpha)$ is no more than $H^{1/2+\alpha}(0,T)$. In fact, though we can establish higher regularity (cf.\cite{li_new_2019}) for the smooth part $	y_s=y-y_0\psi_\lambda$, where
		\[
		\psi_\lambda(t):=1-\frac{\lambda t^\alpha}{\Gamma(\alpha+1)},
		\quad 0\leqslant t\leqslant T,
		%y_0\left(1-\frac{\lambda t^\alpha}{\Gamma(\alpha+1)}\right),
		\]
		the final regularity for $y = y_{s}+ y_0\psi_\lambda$ is dominated by $\psi_\lambda$, which belongs to $H^{1/2+\alpha-\epsilon}(0,T)$ for any $\epsilon>0$ due to the singular term $t^\alpha$.
	\end{rem}
	Combining \cref{lem:decomp-u,lem:regu-y0} gives the following conclusion.
	\begin{thm}\label{thm:regu-u0}
		Assume $f = 0$ and $u_0\in\dot H^{\gamma}(\Omega)$ with $ \gamma\geqslant -1$ if $0<\alpha<1/2$ and $\gamma>1-1/\alpha$, 
		then the weak solution defined by \cref{eq:weak_form} satisfies
		\[
		\begin{split}
		{}&\eta(\epsilon) \nm{u}_{H^{\alpha}(0,T;\dot H^{\beta+\gamma-2-\epsilon}(\Omega))} 
		+\nm{u}_{H^{\alpha/2}(0,T;\dot H^{\beta+\gamma-1}(\Omega))} +
		\eta(\epsilon)\nm{u}_{L^2(0,T;\dot H^{\beta+\gamma-\epsilon}(\Omega))} \\
		\leqslant {}&C_{\alpha,T,\Omega}
		\nm{u_0}_{\dot H^{\gamma}(\Omega)},
		\end{split}
		\]
		where $\beta=\min\{2,1/\alpha\}$, and $\epsilon,\,\eta(\epsilon)$ are defined by \cref{eq:eta-eps}. Moreover, 
				\[
		\left\{
		\begin{aligned}
		&	\epsilon^{3/2}
		\snm{u}_{H^{1/2-\epsilon}(0,T;L^2(\Omega))}
		\leqslant {}
		C_{\alpha,T} \nm{u_0}_{L^{2}(\Omega)},
		&&0<\epsilon\leqslant  1/2,\\
		&		\sqrt{2\gamma-\gamma^2}
		\nm{u}_{H^{(1+\alpha\gamma)/2}(0,T;L^2(\Omega))}
		\leqslant {}C_{\alpha,T,\Omega}\nm{u_0}_{\dot H^{\gamma}(\Omega)},&& 0< \gamma<2,\\
		&		\sqrt{2\gamma-\gamma^2}
		\nm{u}_{H^{(1+\alpha\gamma)/2}(0,T;\dot H^{1}(\Omega))}
		\leqslant {}C_{\alpha,T,\Omega}\nm{u_0}_{\dot H^{\gamma+1}(\Omega)},&& 0< \gamma<2.
		\end{aligned}
		\right.
		\]
	\end{thm}
	\subsection{Regularity in Besov space}
	We now consider the regularity of the solution to \cref{eq:weak_form} in proper Besov spaces. As before, we start from the auxiliary problem \cref{eq:ode-y-f-0} and split it into two cases: $y_0\neq0,\,g=0$; and $y_0=0,\,g=1$. 
	
	To this end, 
	let us present a useful expression of the Mittag-Leffler function \cref{eq:E-ab}; see \cite[Theorem 2.1]{goreno_computation_2002}.
	\begin{lem}[\cite{goreno_computation_2002}]
		\label{eq:ex-Et}
		If $\beta<1$, then for all $0\leqslant t<\infty$, it holds that
		\[
		E_{\alpha,\beta}(-t) = 
		\frac{1}{ \pi  \alpha} \int^{\infty}_{0} 
		\frac{r\sin\beta\pi-t\sin(\alpha-\beta)\pi}{r^2+t^2+2tr\cos\alpha\pi}
		r^{(1-\beta)/\alpha}e^{-r^{1/\alpha}}
		\dd r.
		\]
	\end{lem}
	
	This lemma states that if $\beta<1$, then
	\[
	E_{\alpha ,\beta}(0) = 
	\frac{\sin\beta\pi}{ \pi  } \Gamma(1-\beta).
	\]
	Hence, if $\beta\in\mathbb Z\cap(-\infty,1)$, then $E_{\alpha,\beta}(0)=0$, and thus $E_{\alpha,\beta}(-t)$ is bounded near $t = 0+$. 
	Below, we give a refined estimate that implies the asymptotic behavior of $E_{\alpha,\beta}(-t)$ as $t$ goes to infinity. Note that this estimate has no contraction with the boundness around $t = 0+$ and what we are interested in is the case $t\to\infty$.
	\begin{lem}\label{lem:bd-dk-y}
		If $\beta<1$, then for all $0<t<\infty$, 
		\begin{equation}
		\label{eq:bd-dk-y}
		\snm{E_{\alpha,\beta}(-t)}
		\leqslant C_\alpha 
		\Gamma(1+\theta\alpha-\beta)t^{-\theta},
		\end{equation}
		where $0\leqslant \theta\leqslant 1$.
		Moreover, if $\alpha-\beta\in \mathbb Z$, then \cref{eq:bd-dk-y} 
		holds with $0\leqslant \theta\leqslant 2$.
	\end{lem}
	\begin{proof}
		By \cref{eq:ex-Et}, we have
		\begin{align}
			E_{\alpha,\beta}(-t) = {}&
			\frac{1}{\pi \alpha} 
			\int_{0}^{\infty}
			r^{(1-\beta)/\alpha}
			e^{-r^{1 / \alpha}}  
			\varphi_{\alpha,\beta}(r,t)
			\dd r,\label{eq:Et}
		\end{align}
		where 
		\[
		\varphi_{\alpha,\beta}(r,t): = 
		\frac{r \sin \beta\pi-t\sin (\alpha-\beta)\pi}{r^{2}+2tr  \cos \alpha\pi+t^2}.
		\]
		In light of 
		\begin{equation}\label{eq:est-r-z}
		r^{2}+2tr  \cos \alpha\pi+t^2\geqslant \frac{1+\cos\alpha\pi}{2}(r+t)^2
		\end{equation}
		and the arithmetic-geometric mean inequality (cf. \cite[page 4]{stein_functional_2011})
		\begin{equation}\label{eq:agm}
		r^{1-\theta}t^\theta\leqslant (1-\theta)r+\theta t,\quad 0\leqslant \theta\leqslant 1,
		\end{equation}
		we obtain
		\begin{equation}\label{eq:est-psi}
		\snm{\varphi_{\alpha,\beta}(r,t)}
		\leqslant\frac{C_\alpha}{r+t}
		\leqslant C_\alpha r^{\theta-1}t^{-\theta}.
		\end{equation}
		Therefore, inserting \cref{eq:est-psi} into \cref{eq:Et} gives
		\[
		\begin{split}
		{}&	\snm{	E_{\alpha,\beta}(-t)}
		\leqslant 
		C_{\alpha}t^{-\theta}
		\int_{0}^{\infty}
		e^{-r^{1 / \alpha}}(r^{1/\alpha})^{(\theta-1)\alpha+1-\beta} \dd r\\
		={}&C_\alpha t^{-\theta}
		\int_{0}^{\infty}
		e^{-s}s^{\theta\alpha-\beta} \dd s
		= C_\alpha \Gamma(1+\theta\alpha-\beta)t^{-\theta},
		\end{split}
		\]		
		which establishes \cref{eq:bd-dk-y}. 
		Moreover, if $\alpha-\beta\in\mathbb Z$, then
		\[
		\varphi_{\alpha,\beta}(r,t)=
		\frac{r \sin \beta\pi}{r^{2}+2tr  \cos \alpha\pi+t^2},
		\]
		and 
		from \cref{eq:est-r-z,eq:agm} it follows
		\[
		\begin{split}
		\snm{\varphi_{\alpha,\beta}(r,t)}{}
		\leqslant 	C_\alpha r^{\theta-1}t^{-\theta},
		\quad 0\leqslant \theta\leqslant 2,
		\end{split}
		\]	
		which maintains the estimate \cref{eq:bd-dk-y} and enlarges the range of $\theta$. This completes the proof.
	\end{proof}
	Based on \cref{lem:bd-dk-y}, we are able to establish the following lemma.
	\begin{lem}\label{lem:regu-y-g-0-Besov}
		Assume $-1\leqslant \theta\leqslant 1$ with  $1+2\alpha\theta>0$, 
		then the function $y(t)$ defined by \cref{eq:ode-y0-y} belongs to $\bb_{-\alpha,0}^{1+2\alpha\theta-\epsilon}(0,T)$ with any $0<\epsilon\leqslant 1+2\alpha\theta$ and
		\begin{equation}\label{eq:y-g-0-Besov}
		\nm{y}_{\bb_{-\alpha,0}^{1+2\alpha\theta-\epsilon}(0,T)}
		\leqslant \frac{C_{\alpha,T}}{\sqrt{\epsilon}}\lambda^{\theta}.
		\end{equation}
	\end{lem}
	\begin{proof}
		By \cref{eq:Eabz-est}, it is evident that $y\in L_{\mu^{-\alpha,0}}^2(0,T)$ 
		with the decomposition
		\[
		y:=
		\sum_{k=0}^{\infty}
		y_kS_k^{-\alpha,0},\quad 
		\]
		where $\{S_k^{a,b}\}_{k=0}^\infty$ denotes the shifted 
		Jacobi polynomials on $(0,T)$ with respect 
		to the weight $\mu^{a,b}(t) = (T-t)^at^b$ (see \cref{sec:sJacobi-append}), and
		\begin{equation}\label{eq:yk-xik}
		y_k = \frac{1}{\xi_k^{-\alpha,0}}
		\dual{y,S_k^{-\alpha,0}}_{	\mu^{-\alpha,0}},
		\quad
		\xi_k^{-\alpha,0}
		=\frac{T^{1-\alpha} }{2k+1 -\alpha}.
		\end{equation}
		
		By definition \cref{eq:Habg}, it suffices to investigate the asymptotic behavior of the coefficient $y_k$. Let us fix $k\in\mathbb N_+$. By Rodrigues' formula \cref{eq:St}, we have
		\begin{equation}\label{eq:y-Sk}
		\begin{split}
		\dual{y,S_k^{-\alpha,0}}_{	\mu^{-\alpha,0}}
		&= \dfrac{(-1)^k}{T^{k}k!}  
		\dual{y,\dfrac{\mathrm{d}^{k}}{\dd t^{k}}
			\mu^{k-\alpha,k}}_{(0,T)}.
		\end{split}
		\end{equation}
		Using integration by parts gives
		\[
		\begin{split}
		\dual{y,\dfrac{\mathrm{d}^{k}}{\dd t^{k}}
			\mu^{k-\alpha,k}}_{(0,T)}
		={}&	\sum_{i=0}^{k-1}(-1)^i
		\big(\zeta_i(T)-\zeta_i(0)\big)
		+(-1)^{k}
		\dual{y^{(k)},	\mu^{k-\alpha,k}}_{(0,T)},
		\end{split}
		\]
		where
		\[
		\zeta_i(t)={}	y^{(i)}(t)
		\dfrac{\mathrm{d}^{k-i-1}}{\dd t^{k-i-1}}
		\Big[ \mu^{k-\alpha,k}(t)\Big]  =y^{(i)}(t)
		\sum_{j = 0}^{k-i-1}C_{\alpha,i,j,k}\mu^{k-\alpha-j,i+1+j}(t).
		\]
		By \cref{eq:ode-y0-y} it holds the identity
		\begin{equation}\label{eq:dkv}
		y^{(i)}(t) ={}-\lambda y_0t^{\alpha-i}
		E_{\alpha,\alpha+1-i}(-\lambda t^{\alpha}),
		\quad 0\leqslant i\leqslant k-1,
		\end{equation}
		and it follows that $\zeta_i(0)=\zeta_i(T)=0$ for all $0\leqslant i\leqslant k-1$. Thus, we get the relation
		\begin{equation}\label{eq:yk}
		\dual{y,\dfrac{\mathrm{d}^{k}}{\dd t^{k}}
			\mu^{k-\alpha,k}}_{(0,T)}={}	(-1)^{k}
		\dual{y^{(k)},	\mu^{k-\alpha,k}}_{(0,T)}.
		\end{equation}
		Invoking \cref{lem:bd-dk-y} yields the inequality
		\begin{equation}
		\label{eq:est-dk-y-homo}
		\begin{split}
		\snm{y^{(k)}(t)}\leqslant {}&C_{\alpha}\lambda \snm{y_0}t^{\alpha-k}
		\snm{	E_{\alpha,\alpha+1-k}(-\lambda t^{\alpha})}
		\leqslant {}C_{\alpha}	\Gamma(k-\theta\alpha)
		\snm{y_0}\lambda^{\theta}
		t^{\theta\alpha-k},
		\end{split}
		\end{equation}
		where $-1\leqslant \theta\leqslant 1$.
		In addition, we have a
		useful formula \cite[Appendix, (A.6)]{shen_spectral_2011}
		\begin{equation}\label{eq:int-Bab}
		\nm{\mu^{a,b}}_{L^1(0,T)} 
		=T^{a+b+1}
		\frac{\Gamma(a+1)\Gamma(b+1)}{\Gamma(a+b+2)},
		\quad a,b>-1,
		\end{equation}
		which, together with  \cref{eq:est-dk-y-homo}, indicates
		that
		\begin{equation}\label{eq:ydk-w}
		\begin{split}
		\snm{		\dual{y^{(k)},	\mu^{k-\alpha,k}}_{(0,T)}}
		\leqslant {}&
		C_{\alpha} \snm{y_0}\lambda^{\theta}
		\Gamma(k-\theta\alpha)
		\dual{1,\mu^{k-\alpha,\theta\alpha}}_{(0,T)} \\
		={}&C_{\alpha} \snm{y_0} \lambda^{\theta}
		\frac{\Gamma(k+1-\alpha)	\Gamma(k-\theta\alpha)}{	\Gamma(k+2+(\theta-1)\alpha)}
		T^{k+1+(\theta-1)\alpha}.
		\end{split}
		\end{equation}
		For the Gamma function \cref{eq:Gamma}, we have the Stirling's formula \cite[Eq. (6.1.38)]{Abramovitz1972}
		\begin{equation}\label{eq:str-pos}
		\Gamma(z+1)=\sqrt{2 \pi} z^{z+1 / 2} e^{\theta(z)/(12z)-z},
		\quad 0<\theta(z)<1,\,z>0.
		\end{equation}
		Therefore, collecting \cref{eq:y-Sk,eq:yk,eq:str-pos} gives
		\begin{equation}\label{eq:yk-est}
		\begin{split}
		{}&	\snm{	\dual{y,S_k^{-\alpha,0}}_{	\mu^{-\alpha,0}}}=
		\frac{1}{T^{k}k!}\snm{	\dual{y,\dfrac{\mathrm{d}^{k}}{\dd t^{k}}
				\mu^{k-\alpha,k}}_{(0,T)}}\\
		\leqslant {}&
		C_{\alpha,T} \snm{y_0}\lambda^{\theta}
		\frac{\Gamma(k+1-\alpha)	\Gamma(k-\theta\alpha)}{k!	\Gamma(k+2+(\theta-1)\alpha)}
		\leqslant {}
		C_{\alpha,T}\snm{y_0} \lambda^{\theta}
		k^{-2-2\theta\alpha}.
		\end{split}
		\end{equation}
		%Consequently, it follows that
		As a result,
		\[
		\begin{split}
		\nm{y}_{\bb_{-\alpha,0}^{1+2\alpha\theta-\epsilon}(0,T)}^2={}&
		\sum_{k = 0}^{\infty} 
		\frac{1+k^{2+4\alpha\theta-2\epsilon}}{\xi_k^{-\alpha,0}}
		\snm{	\dual{y,S_k^{-\alpha,0}}_{	\mu^{-\alpha,0}}}^2\\
		\leqslant {}&	 	C_{\alpha,T}\snm{y_0}^2 \lambda^{2\theta}
		\left(1+	\sum_{k = 1}^{\infty} k^{-1-2\epsilon} \right)\\
		\leqslant {}&	 	C_{\alpha,T}\snm{y_0}^2 \lambda^{2\theta}
		\int_{1}^{\infty}r^{-1-2\epsilon}\dd r = \frac{C_{\alpha,T}}{\epsilon}
		\snm{y_0}^2 \lambda^{2\theta},
		\end{split}
		\]
		which showes \cref{eq:y-g-0-Besov} and finishes the proof.
	\end{proof}
	\begin{rem}\label{rem:comp-y-Sobolev-Besov}
		For a fixed $\lambda>0$, $y(t)=y_0E_{\alpha,1}(-\lambda t^\alpha)$ 
		has a leading singular term $t^\alpha$, and as mentioned in \cref{rem:regu-ML}, the highest regularity of $y(t)$ is no more than $H^{1/2+\alpha}(0,T)$. However, from \cref{lem:regu-y-g-0-Besov} we observe that   $y\in\bb_{-\alpha,0}^{1+2\alpha-\epsilon}(0,T)$, which can not be improved due to the singular term $t^\alpha$, and the optimal rate $1+2\alpha$ of the standard Legendre spectral method under $L^2$-norm has been validated numerically in \cref{fig:ode_test}. Unfortunately, if $\lambda$ is extremely large or goes to infinity, then we see from \cref{eq:y-g-0-Besov} that the regularity and convergence rate will be ruined (we refer the reader to \cite[Section 1.2]{duan_exponentially_2019} for detailed numerical investigations in this situation). 	
	\end{rem}
	For the particular case $y_0=0$ and $g=1$, the solution to the auxiliary 
	problem \cref{eq:ode-y-f-0} is given by (cf. \cite[Theorem 5.4]{kilbas_theory_2006})
	\begin{equation}\label{eq:y-y0-0-Besov}
	y(t) = t^\alpha E_{\alpha,\alpha+1}(-\lambda t^\alpha),
	\quad 0\leqslant t\leqslant T,
	\end{equation}
	which means, for all $k\in\mathbb N$, %a direct computation gives
	\begin{equation*}
		y^{(k)}(t) ={}t^{\alpha-k}
		E_{\alpha,\alpha+1-k}(-\lambda t^{\alpha}).
		\quad 0<t\leqslant T,
	\end{equation*}
	Then invoking \cref{lem:bd-dk-y} yields the estimate
	\begin{equation}\label{eq:dk-y}
	\begin{split}
	\snm{y^{(k)}(t)}\leqslant {}&
	C_{\alpha} \lambda^{\theta-1}
	\Gamma(k-\theta\alpha)
	t^{\theta\alpha-k},
	\quad 0<t\leqslant T,
	\end{split}
	\end{equation}
	where $-1\leqslant \theta\leqslant 1$. Hence, analogously to 
	\cref{lem:regu-y-g-0-Besov}, we can prove the following result.
	\begin{lem}\label{lem:regu-y-Besov}
		Assume $-1\leqslant \theta\leqslant 1$ with $1+2\alpha\theta>0$, 
		then the function $y(t)$ in \cref{eq:y-y0-0-Besov} belongs to $\bb_{-\alpha,0}^{1+2\alpha\theta-\epsilon}(0,T)$ with any $0<\epsilon\leqslant 1+2\alpha\theta$ and
		\[
		\nm{y}_{\bb_{-\alpha,0}^{1+2\alpha\theta-\epsilon}(0,T)}
		\leqslant \frac{C_{\alpha,T}}{\sqrt{\epsilon}}\lambda^{\theta-1}.
		\]
	\end{lem}

	Finally, gathering \cref{lem:decomp-u}, \cref{lem:regu-y-g-0-Besov} and \cref{lem:regu-y-Besov} 
	implies the following regularity result in the Besov space.
	\begin{thm}\label{thm:Besov-regu-u-f-0}
		If $f= 0$ and $u_0\in \dot H^\gamma(\Omega)$ with $\max\{-1,1-1/\alpha\}<\gamma\leqslant 3$, then the weak solution 
		defined by \cref{eq:weak_form} belongs to $ \bb_{-\alpha,0}^{1+\alpha(\gamma-\beta)-\epsilon}(0,T;\dot H^{\beta}(\Omega))$ with 
		\begin{equation}\label{eq:regu-besov}
		\nm{u}_{\bb_{-\alpha,0}^{1+\alpha(\gamma-\beta)-\epsilon}(0,T;\dot H^{\beta}(\Omega))}
		\leqslant \frac{C_{\alpha,T}}{\sqrt{\epsilon}}
		\nm{u_0}_{\dot H^\gamma(\Omega)},
		\end{equation}
		where $\gamma-2\leqslant \beta\leqslant \gamma+2$, $\beta<\gamma+1/\alpha$, and  $0<\epsilon\leqslant 1+\alpha(\gamma-\beta)$. 
	\end{thm}
%	\begin{rem}
%		If $0\leqslant \gamma\leqslant 2$, then by \cref{eq:regu-besov} we have
%		\[
%		\begin{aligned}
%		\nm{u}_{\bb_{-\alpha,0}^{1+\alpha\gamma -\epsilon}(0,T;L^{2}(\Omega))}
%		\leqslant {}&\frac{C_{\alpha,T}}{\sqrt{\epsilon}}
%		\nm{u_0}_{\dot H^\gamma(\Omega)},\\
%		\nm{u}_{\bb_{-\alpha,0}^{1+\alpha\gamma -\epsilon}(0,T;\dot H^{1}(\Omega))}
%		\leqslant {}&\frac{C_{\alpha,T}}{\sqrt{\epsilon}}
%		\nm{u_0}_{\dot H^{\gamma+1}(\Omega)},\\
%		\end{aligned}
%		\]
%		for any $0<\epsilon\leqslant 1+\alpha\gamma$. 
%		Comparing these estimates with \cref{eq:temp-regu-u0}, we find that the index of the 
%		temporal regularity in the each Besov space is almost twice as that in the corresponding Sobolev space.
%	\end{rem}
	\begin{thm}\label{thm:Besov-regu-u-u0-0}
		If $u_0=0$ and $f = v\in\dot H^{\gamma}(\Omega)$ with $\gamma\geqslant -1$, then the weak solution 
		defined by \cref{eq:weak_form} belongs to $ \bb_{-\alpha,0}^{1+\alpha(2+\gamma-\beta)-\epsilon}(0,T;\dot H^{\beta}(\Omega))$ with 
		\[
		\nm{u}_{\bb_{-\alpha,0}^{1+\alpha(2+\gamma-\beta)-\epsilon}(0,T;\dot H^{\beta}(\Omega))}
		\leqslant \frac{C_{\alpha,T}}{\sqrt{\epsilon}}
		\nm{v}_{\dot H^\gamma(\Omega)},
		\]
		where $\gamma\leqslant \beta\leqslant 4+\gamma$, $\beta<2+\gamma+1/\alpha$, and $0<\epsilon\leqslant 1+\alpha(2+\gamma-\beta)$. 
	\end{thm}
	\section{Discretization and  Error Analysis}
	\label{sec:main}
	%To the end of this section, we present our numerical scheme. 
	%Assume $M\in\mathbb N$. 
	Let  $\mathcal K_h $ be a conventional conforming and 
	shape regular triangulation of $ \Omega $
	consisting of $ d $-simplexes, and let $ h $  denote the maximum diameter of the
	elements in $ \mathcal K_h $. Define
	\[
	X_h:= \left\{
	v_h \in \dot H^1(\Omega):\
	v_h|_K \in P_1(K)\quad\forall \,K
	\in \mathcal K_h
	\right\}.
	\]
	For $M\in\mathbb N$, our time-spectral method for problem \cref{eq:model} 
	reads as follows: find $ U \in P_M(0,T)\otimes X_h$ such that
	\begin{equation}
	\label{eq:U}
	\dual{\D_{0+}^\alpha U,V}_{H^{\alpha/2}(0,T;L^2(\Omega))} +
	\dual{\nabla U,\nabla V}_{\Omega\times(0,T)} =
	\dual{f+\D_{0+}^\alpha u_0,V}_{P_M(0,T)\otimes X_h}
	\end{equation}
	for all $ V \in P_M(0,T)\otimes X_h $. It is easy to see that 
	%The stability result is similar with that given in \cite{li_luo_xie_analysis_2019} for a discontinuous Galerkin method, i.e.,
	\cref{eq:U} admits a unique solution 
	$U\in P_M(0,T)\otimes X_h$ such that
	\[
	\nm{U}_{\mathcal X} \leqslant C_{\alpha}
	\nm{f+\D_{0+}^\alpha u_0 }_{\mathcal X^*},
	\]
	where the space $\mathcal X$ and its norm $\nm{\cdot}_{\mathcal X}$ are defined in \cref{sec:regu}.
	
	We now present our main error estimates. Note that all of the following results are optimal with respect to the solution regularity. 
%	Moreover, temporal convergence rates $1+\alpha$ and $1+2\alpha$ under $H^{\alpha/2}(0,T;L^2(\Omega))$-norm and $L^2(0,T;\dot H^1(\Omega))$-norm are obtained, respectively. 
	Also, we emphasis that, in \cref{thm:est-u0,thm:est-f}, the best rates $1+\alpha$ and $1+2\alpha$ under $H^{\alpha/2}(0,T;L^2(\Omega))$-norm and $L^2(0,T;\dot H^1(\Omega))$-norm are sharp and cannot be improved 
	even for smoother $u_0$ and $v$, respectively. This is verified by the numerical results in \cref{sec:numer}.
	%, all of which are optimal with 
	%respect to the regularity.
	%	If $f+\D_{0+}^\alpha u_0\in\mathcal X^*$, then we have
	%\begin{equation}
	%\nm{u-U}_{\mathcal X}
	%\lesssim{}\nm{(I-P_h)u}_{\mathcal X}+
	%\nm{(I-\Phi_{M}^{-\alpha,0})u}_{\mathcal X}
	%:={}\mathcal E(u).
	%\end{equation}	
	%	If $f+\D_{0+}^\alpha u_0\in\mathcal X^*$, then we have
	%\begin{equation}
	%\begin{split}
	%\nm{u-U}_{L^2(0,T;L^2(\Omega))}^2 
	%\lesssim{}&
	%(h+M^{-\alpha})^2\mathcal E^2(u)+\mathcal E(u)\mathcal T(u),
	%\end{split}
	%\end{equation}
	%where $	\mathcal E(u)$ is defined in \cref{eq:err-L2H1} and 
	%\[
	%\mathcal T(u):=(h+M^{-\alpha})
	%\nm{(I-\Phi_{M}^{\alpha,0})u}_{L^2(0,T;L^2(\Omega))}+
	%\nm{\big(I-\Phi_M^{\alpha,0}\big)u}_{H^{-\alpha/2}(0,T;L^2(\Omega))}.
	%\]
	\begin{thm}\label{thm:est-u}
		If the solution to problem \cref{eq:model} is of the form $u(x,t) = t^\beta\phi(x)$ 
		with $\beta>(\alpha-1)/2$ and $\phi\in\dot H^{2}(\Omega)$, then 
		\begin{align}
			\nm{u-U}_{L^{2}(0,T;\dot H^1(\Omega))} 
			\lesssim {}&\left(h+M^{-1-2\beta}\right) \nm{\phi}_{\dot H^{2}(\Omega)},
			\label{eq:est-u-L2H1}\\
			\nm{u-U}_{H^{\alpha/2}(0,T;L^2(\Omega))} 
			\lesssim {}&\left(h^2+M^{\alpha-1-2\beta}\right) 
			\nm{\phi}_{\dot H^{2}(\Omega)}.
			\label{eq:est-u-HsL2}
		\end{align}
	\end{thm}
	\begin{thm}
		\label{thm:est-u0}
		If $f=0$ and $u_0\in\dot H^\gamma(\Omega)$ with $\max\{-1,1-1/\alpha\}<\gamma\leqslant 2$, then
		%\[
		%\begin{split}
		%%\nm{u-U}_{L^{2}(0,T;\dot H^1(\Omega))} 
		%%\lesssim {}&\left(\epsilon_hh^{\min\{1/\alpha,2\}+\gamma-1} +M^{-1-\alpha(\gamma-1)}\right) \nm{u_0}_{\dot H^{\gamma}(\Omega)},\quad \theta = \frac{\gamma-1}{2}\\
		%\nm{u-U}_{H^{\alpha/2}(0,T;L^2(\Omega))} 
		%\lesssim {}&\left(h^{\min\{1/\alpha,2\}+\gamma-1}
		%+M^{-1-\alpha(\gamma-1)}\right) 
		%\nm{u_0}_{\dot H^{\gamma}(\Omega)}.
		%\quad \theta = \frac{\gamma}{2}
		%\end{split}
		%\]
		\begin{equation}
		%		\nm{u-U}_{L^{2}(0,T;\dot H^1(\Omega))} 
		%	\lesssim {}&\left(h^{1+\gamma}+M^{-1-\alpha(\gamma+1)}\right) 
		%	\nm{v}_{\dot H^{\gamma}(\Omega)},
		%	\label{eq:est-f-L2H1}\\
		\nm{u-U}_{H^{\alpha/2}(0,T;L^2(\Omega))} 
		\lesssim {}\left(h^{\min\{2,\gamma+1,1/\alpha+\gamma-1\}}
		+M^{-1-\alpha(\gamma-1)}\right) 
		\nm{u_0}_{\dot H^{\gamma}(\Omega)}.
		\label{eq:est-u0-HsL2}
		\end{equation}
		Moreover, if $u_0\in \dot H^\gamma(\Omega)$ with $\max\{-1,1-1/\alpha\}<\gamma\leqslant 3$, then 
		\begin{equation}
		\label{eq:est-u0-L2H1}
		\nm{u-U}_{L^{2}(0,T;\dot H^1(\Omega))} 
		\lesssim {}\left(\epsilon_hh^{\min\{1,\gamma+1,1/\alpha+\gamma-1\}} +M^{-1-\alpha(\gamma-1)}\right) \nm{u_0}_{\dot H^{\gamma}(\Omega)},
		\end{equation}
		%
		%\begin{align}
		%\nm{u-U}_{L^{2}(0,T;\dot H^1(\Omega))} 
		%\lesssim {}&\left(h+M^{-1-2\gamma}\right) \nm{\phi}_{\dot H^{2}(\Omega)},
		%\label{eq:est-u0-L2H1}\\
		%\nm{u-U}_{H^{\alpha/2}(0,T;L^2(\Omega))} 
		%\lesssim {}&\left(h^2+M^{\alpha-1-2\gamma}\right) 
		%\nm{\phi}_{\dot H^{2}(\Omega)}.
		%\label{eq:est-u0-HsL2}
		%\end{align}
		where $\epsilon_h=1$ if $\alpha\neq1/2$, and $\epsilon_h=\sqrt{\snm{\ln h}}$ if $\alpha=1/2$.
		%Moreover, if $u_0\in\dot H^\gamma(\Omega)$ with $0\leqslant \gamma\leqslant 2$, then
		%\[
		%\nm{u-U}_{H^{\alpha/2}(0,T;L^2(\Omega))}
		%+\nm{u-U}_{L^2(0,T;\dot H^1(\Omega))}
		%\lesssim{}\left(M^{-1-\alpha(\gamma-1)}+h 
		%\right)
		%\nm{u_0}_{\dot H^{\gamma}(\Omega)}.
		%\]
	\end{thm}
	%In the following, we present several superconvergent 
	%results with $f=0$ and nonsmooth initial data $u_0\in\dot H^{\gamma}(\Omega),\,-1\leqslant \gamma\leqslant 2$.
	\begin{thm}\label{thm:est-f}
		If $u_0=0$ and $f(t,x)=v(x)\in \dot H^\gamma(\Omega)$ with $-1\leqslant \gamma\leqslant 0$, then 
		\begin{equation*}
			\nm{u-U}_{H^{\alpha/2}(0,T;L^2(\Omega))} 
			\lesssim {}\left(h^{\gamma+2}+M^{-1-\alpha(\gamma+1)}\right) 
			\nm{v}_{\dot H^{\gamma}(\Omega)}.
			%	\label{eq:est-f-HsL2}
		\end{equation*}
		In addition, if $v\in \dot H^\gamma(\Omega)$ with $-1\leqslant \gamma\leqslant 1$, then 
		\begin{equation*}
			%		\label{eq:est-f-L2H1}
			\nm{u-U}_{L^2(0,T;\dot H^1(\Omega))}
			\lesssim{}\left(h^{\min\{1,\gamma+1\}}
			+M^{-1-\alpha(\gamma+1)}\right)
			\nm{v}_{\dot H^{\gamma}(\Omega)}.
		\end{equation*}
	\end{thm}
	\subsection{Technical lemmas}
	%In the sequential part, we are going to prove the above theorems.
	Let $R_h:\dot H^1(\Omega)\longrightarrow X_h$ 
	be the well-known Ritz projection operator defined by
	\[
	\dual{\nabla(I-R_h)v,\nabla v_h} = 0\quad \forall\,v_h\in X_h,
	\]
	for which we have the standard estimate \cite{quarteroni_numerical_2008}
	\[
	\nm{(I-R_h)v}_{L^2(\Omega)}+h\nm{(I-R_h)v}_{\dot H^1(\Omega)}
	\lesssim h^{\gamma}\nm{v}_{\dot H^\gamma(\Omega)}
	\quad\forall\,v\in \dot H^\gamma(\Omega),
	\quad 1\leqslant \gamma\leqslant 2.
	\]
	
	In what follows, we give some 
%	Thanks to \cref{lem:decomp-u}, the error estimation for
%	the temporary discretization is reduced to proving some 
	nontrivial	projection error bounds in terms of the solution to \cref{eq:ode-y-f-0}.  Of course, it is worth noticing that
	one can use the regularity results   in \cref{lem:regu-y-g-0-Besov,lem:regu-y-Besov} 
	together with standard projection error estimates \cite{guo_jacobi_2004} and interpolation techniques \cite{tartar_introduction_2007}
	to obtain quasi-optimal results (with logarithm factors). But here we can prove optimal error 
	estimates by recycling the proof of \cref{lem:regu-y-g-0-Besov} and extending 
	it to the estimation under the fractional norm $\nm{\cdot}_{H^{\alpha/2}(0,T)}$.
	\begin{lem}\label{lem:Phi-L2-w}
		Let $y(t)$ be given by \cref{eq:ode-y0-y} with $\lambda>0$ and $y_0\in\R$, then
		\begin{equation}
		\label{eq:Phi-L2}
		\nm{\big(I-\Phi_M^{-\alpha,0}\big)y}_{L^2_{\mu^{-\alpha,0}}(0,T)}
		\leqslant {}C_{\alpha,T} \snm{y_0}
		\lambda^{\theta}M^{-1-2\alpha\theta},
		\end{equation}
		where $-1\leqslant \theta\leqslant 1$ and $1+2\alpha\theta>0$.
	\end{lem}
	\begin{proof}	
		It is evident that $y\in L_{\mu^{-\alpha,0}}^2(0,T)$  and that the $L^2_{\mu^{-\alpha,0}}$-orthogonal projection of $y$ onto $P_M(0,T)$ is given by
		\[
		\Phi_{M}^{-\alpha,0}y:=
		\sum_{k=0}^{M}
		y_kS_k^{-\alpha,0},\quad 
		\]
		where
		\begin{equation}
		\label{eq:yk-lem-41}
		y_k = 
%		T^{\alpha-1}(2k+1 -\alpha)
		\frac{2k+1 -\alpha}{T^{1-\alpha}}
		\dual{y,S_k^{-\alpha,0}}_{	\mu^{-\alpha,0}}.
		%	\quad
		%	\xi_k^{-\alpha,0}
		%	=\frac{T^{1-\alpha} }{2k+1 -\alpha}.
		\end{equation}
		Thus we obtain
		\begin{equation*}
			%	\label{eq:piM-L2-w}
			\nm{\big(I-\Phi_M^{-\alpha,0}\big)y}^2_{L^2_{\mu^{-\alpha,0}}(0,T)}
			=T^{\alpha-1} \sum_{k=M+1}^{\infty}
			(2k+1 -\alpha)\snm{	\dual{y,S_k^{-\alpha,0}}_{	\mu^{-\alpha,0}}}^2,
		\end{equation*}
		Thanks to the estimate \cref{eq:yk-est}, we get
		\begin{equation}
		\label{eq:est-L2}
		\begin{split}
		{}&		\nm{\big(I-\Phi_M^{-\alpha,0}\big)y}_{L^2_{\mu^{-\alpha,0}}(0,T)}^2
		\leqslant 
		C_{\alpha,T}
		\snm{y_0}^2\lambda^{2\theta}
		\sum_{k = M+1}^{\infty} 	k^{-3-4\theta\alpha}\\
		\leqslant {}&		C_{\alpha,T}
		\snm{y_0}^2\lambda^{2\theta}
		\int_{M}^{\infty} r^{-3-4\theta\alpha}\dd r = 	
		C_{\alpha,T}
		\snm{y_0}^2\lambda^{2\theta} M^{-2-4\theta\alpha},
		\end{split}
		\end{equation}
		which establishes \cref{eq:Phi-L2} and finishes the proof of this lemma.
	\end{proof}
	\begin{rem}\label{rem:L2}
		Observing the standard inequality
		\[
		\nm{(I-\Phi_M^{-\alpha,0})y}_{L^2(0,T)}
		\leqslant {}C_{\alpha,T}
		\nm{(I-\Phi_M^{-\alpha,0})y}_{L^2_{\mu^{-\alpha,0}}(0,T)},
		\]
		we conclude from \cref{lem:Phi-L2-w} that
		\begin{equation*}
			%	\label{eq:Phi-est-L2}
			\nm{(I-\Phi_M^{-\alpha,0})y}_{L^2(0,T)}
			\leqslant {}C_{\alpha,T} 	\snm{y_0}\lambda^{\theta}
			M^{-1-2\alpha\theta},
		\end{equation*}
		where $-1\leqslant \theta\leqslant 1$ and $1+2\alpha\theta>0$.
	\end{rem}
	\begin{lem}\label{lem:Phi-est-Hs}
		Let $y(t)$ be given by \cref{eq:ode-y0-y} with $\lambda>0$ and $y_0\in\R$, then
		\begin{align}
			\label{eq:Phi-est-Hs}
			%	\inf_{q\in P_M(0,T)}
			%	\snm{y-q}_{H^{\alpha/2}(0,T)}={}&
			\snm{(I-\Phi_M^{-\alpha,0})y}_{H^{\alpha/2}(0,T)} 
			\leqslant {}&C_{\alpha,T} \snm{y_0}
			\lambda^{\theta}M^{\alpha-1-2\alpha\theta},
			%		\label{eq:Phi-est-Hms}
			%		%	\inf_{q\in P_M(0,T)}
			%		%	\snm{y-q}_{H^{-\alpha/2}(0,T)}={}&
			%		\snm{\big(I-\Phi_M^{\alpha,0}\big)y}_{H^{-\alpha/2}(0,T)} 
			%		\leqslant {}&C_{\alpha,T} \snm{y_0}
			%		\lambda^{\theta}M^{-\alpha-1-2\alpha\theta},
		\end{align}
		where $-1\leqslant \theta\leqslant 1$ and $1+2\alpha\theta>\alpha$.
	\end{lem}
	\begin{proof}
		%	It is evident that $v\in L_{\mu^{-\alpha,0}}^2(0,T)$  and 
		%	\[
		%	\Phi_{M}^{-\alpha,0}v:=
		%	\sum_{k=0}^{M}
		%	y_kS_k^{-\alpha,0},\quad 
		%	%	y_k = \frac{1}{\xi_k^{-\alpha,0}}
		%	%	\int_0^T v(t)S_{k}^{-\alpha,0}(t) 
		%	%	\mu^{-\alpha,0}(t) \dd t,
		%	%		\sum_{k=0}^{\infty}
		%	%	y_kS_k^{-\alpha,0},\quad 
		%	y_k = \frac{1}{\xi_k^{-\alpha,0}}
		%	%	\int_{0}^{T} v(t)S_{k}^{-\alpha,0}(t) 
		%	%	\mu^{-\alpha,0}(t) \dd t
		%	\dual{v,S_k^{-\alpha,0}}_{	\mu^{-\alpha,0}},
		%	%	\quad
		%	%		\xi_k^{-\alpha,0}
		%	%	=\frac{T^{1-\alpha} }{2k+1 -\alpha}.
		%	\]
		%	where
		%	\[
		%	%	\xi_k^{0,-\alpha}=
		%	\xi_k^{-\alpha,0}
		%	=\frac{T^{1-\alpha} }{2k+1 -\alpha}.
		%	\]
		Set 
		\[
		z :=(I-\Phi_M^{-\alpha,0})y = 
		\sum_{k=M+1}^{\infty}
		y_kS_k^{-\alpha,0},
		\]
		where $y_k$ is given by \cref{eq:yk-lem-41}. Since $z\in L_{\mu^{0,-\alpha}}^2(0,T)$, we have the orthogonal expansion
		\[
		z=	\sum_{k=0}^{\infty}z_kS_k^{0,-\alpha} \quad\text{with }
		z_k = {} \frac{2k+1 -\alpha}{T^{1-\alpha}}
		\dual{z,S_k^{0,-\alpha}}_{	\mu^{0,-\alpha}}.
		\]
		%	where 
		%	\[
		%		g_k = {}\frac{1}{\xi_k^{0,-\alpha}}
		%	\dual{g,S_k^{0,-\alpha}}_{	\mu^{0,-\alpha}},
		%	\quad 	\xi_k^{0,-\alpha}
		%	%	\xi_k^{-\alpha,0}
		%	=\frac{T^{1-\alpha} }{2k+1 -\alpha}.
		%	\]
		%	where
		%	\[
		%	\xi_k^{0,-\alpha}
		%	%	\xi_k^{-\alpha,0}
		%	=\frac{T^{1-\alpha} }{2k+1 -\alpha}.
		%	\]
		Thanks to \cref{eq:DsS}, we see that
		\[
		%\begin{split}
		\D_{0+}^{\alpha/2}z ={}
		\sum_{k=M+1}^{\infty}
		\frac{	y_k k!}{\Gamma(k+1-\alpha/2)}
		t^{-\alpha/2}
		S_k^{-\alpha/2},\]
		\[
		\D_{T-}^{\alpha/2}z ={}
		\sum_{k=0}^{\infty}
		\frac{	z_k k!	(T-t)^{-\alpha/2}}{\Gamma(k+1-\alpha/2)}
		S_k^{-\alpha/2}.
		%\end{split}
		\]
		%		We first prove \cref{eq:Phi-est-Hs}. By \cref{lem:coer}, we see that
		%	\begin{equation}\label{eq:Hs-Ds}
		%	\cos(\alpha\pi/2)	|v|_{H^{\alpha/2}(0,T)}^2
		%	=	
		%	\dual{\D_{0+}^{\alpha/2} v,\D_{T-}^{\alpha/2} v}_{(0,T)}
		%	= 
		%	\dual{\D_{0+}^{\alpha} v,v}_{H^{\alpha/2}(0,T)},
		%	\end{equation}
		%	where $v\in H^{\alpha/2}(0,T)$ is arbitrary. 
		Hence, from \cref{lem:coer} and the orthogonality of $\{S_k^{-\alpha/2}\}_{k=0}^\infty$ with respect to the weight  $\mu^{-\alpha/2}$,  it follows that
		\begin{align}
			\cos(\alpha\pi/2)	|z|_{H^{\alpha/2}(0,T)}^2
			=	{}&	
			\dual{\D_{0+}^{\alpha/2} z,\D_{T-}^{\alpha/2} z}_{(0,T)}
			%	={}\sum_{k=M+1}^{\infty}	\frac{	y_kg_k\Gamma^2(k+1)}{\Gamma^2(k+1-\alpha/2)}
			%	\xi_k^{-\alpha/2}	\notag\\
			={}\sum_{k=M+1}^{\infty}
			\frac{T^{1-\alpha}y_kz_k\Gamma(k+1)}
			{(2k+1-\alpha) \Gamma(k+1-\alpha)}.
			\label{eq:Dsg-Dsg}
		\end{align}
		
		By the Rodrigues' formula \cref{eq:St}, we have
		%	\[
		%	\mu^{a,b}(t)S_{k}^{a,b}(t)=\frac{(-1)^{k}}{T^{k} k !}
		%	\frac{\dd^{k}}{\dd t^{k}}\mu^{k+a,k+b}(t).
		%	\]
		\[
		\begin{split}
		y_k ={}&
		%	T^{\alpha-1}(2k+1-\alpha)	\dual{v,S_k^{-\alpha,0}}_{	\mu^{-\alpha,0}}= 
		(-1)^k\frac{2k+1 -\alpha}{T^{k+1-\alpha} k!}
		\dual{y,\dfrac{\mathrm{d}^{k}}{\dd t^{k}}
			\mu^{k-\alpha,k}}_{(0,T)},\\
		z_k ={}&
		%	T^{\alpha-1}(2k+1-\alpha)	\dual{g,S_k^{0,-\alpha}}_{	\mu^{0,-\alpha}}= 
		(-1)^k\frac{2k+1 -\alpha}{T^{k+1-\alpha} k!}
		\dual{z,\dfrac{\mathrm{d}^{k}}{\dd t^{k}}
			\mu^{k,k-\alpha}}_{(0,T)}.
		\end{split}
		\]
		It follows from \cref{eq:y-Sk,eq:yk} that
		\[
		\dual{y,\dfrac{\mathrm{d}^{k}}{\dd t^{k}}
			\mu^{k-\alpha,k}}_{(0,T)}
		= {}
		(-1)^k
		\dual{y^{(k)},\mu^{k-\alpha,k}}_{(0,T)}.
		\]
		When $k\geqslant M+1$, it holds $z^{(k)}(t) = y^{(k)}(t)$. 
		Therefore, applying the proof of \cref{eq:yk} gives
		\[
		\dual{z,\dfrac{\mathrm{d}^{k}}{\dd t^{k}}
			\mu^{k,k-\alpha}}_{(0,T)}
		= {}
		(-1)^k
		\dual{y^{(k)},\mu^{k,k-\alpha}}_{(0,T)}.
		\]
		Substituting the above two equalities into \cref{eq:Dsg-Dsg} implies
		\begin{equation}\label{eq:g}
		\cos(\alpha\pi/2)	|z|_{H^{\alpha/2}(0,T)}^2
		=
		\sum_{k=M+1}^{\infty}
		\frac{2k+1 -\alpha}{T^{1-\alpha}}\cdot
		\frac{\dual{y^{(k)},\mu^{k-\alpha,k}}_{(0,T)}}{ T^k\Gamma(k+1-\alpha)}
		\cdot
		\frac{\dual{y^{(k)},\mu^{k,k-\alpha}}_{(0,T)}}{T^k\Gamma(k+1)}.
		\end{equation}
		From the estimate \cref{eq:ydk-w}, we obtain
		\begin{equation}\label{eq:est-Hs-1}
		\begin{split}
		\frac{	\snm{\dual{y^{(k)},\mu^{k-\alpha,k}}_{(0,T)}}}{T^k \Gamma(k+1-\alpha)}
		\leqslant {}&
		C_{\alpha} \snm{y_0}\lambda^{\beta}
		\frac{	\Gamma(k-\beta\alpha)}{ T^k\Gamma(k+1-\alpha)}
		\nm{\mu^{k-\alpha,\beta\alpha}}_{L^1(0,T)}\\
		={}&	C_{\alpha}\snm{y_0}\lambda^{\beta}T^{1+(\beta-1)\alpha}
		\frac{\Gamma(k-\beta\alpha)}{\Gamma(k+2+(\beta-1)\alpha)},
		\end{split}
		\end{equation}
		where $-1\leqslant \beta\leqslant 1$. 
		Besides, based on \cref{eq:est-dk-y-homo,eq:int-Bab}, a 
		similar manipulation as that of \cref{eq:ydk-w} indicates
		\begin{equation}\label{eq:est-Hs-2}
		\begin{split}
		\frac{\snm{	\dual{y^{(k)},\mu^{k,k-\alpha}}_{(0,T)}}}{T^k\Gamma(k+1)}
		\leqslant {}&
		C_{\alpha} \snm{y_0}\lambda^{\gamma}
		\frac{\Gamma(k-\gamma\alpha)}{T^k\Gamma(k+1)}
		\nm{\mu^{k,(\gamma-1)\alpha}}_{L^1(0,T)}\\
		={}&C_{\alpha}\snm{y_0}\lambda^{\gamma}T^{1+(\gamma-1)\alpha}
		\frac{\Gamma(k-\gamma\alpha)}{\Gamma(k+2+(\gamma-1)\alpha)},
		\end{split}
		\end{equation}
		where $-1\leqslant \gamma\leqslant 1$.
		%	 such that $1+2\alpha\theta>0$.
		
		%	Set $\gamma = (\beta+\theta)/2$ and let $1+2\alpha\gamma>\alpha$.
		%	 such that $1+2\alpha\beta>0$.
		Consequently, plugging \cref{eq:est-Hs-1,eq:est-Hs-2} into \cref{eq:g} 
		and applying Stirling's formula \cref{eq:str-pos} yield
		\[
		\begin{split}
		|z|_{H^{\alpha/2}(0,T)}^2
		\leqslant {}&C_{\alpha,T} 
		\snm{y_0}^2
		\lambda^{\beta+\gamma}
		\sum_{k=M+1}^{\infty}
		%\frac{
		(	2k+1-\alpha)
		%}{\Gamma(k+2+(\beta-1)\alpha)}
		%\cdot
		\frac{\Gamma(k-\beta\alpha)}
		{\Gamma(k+2+(\beta-1)\alpha)}
		\cdot
		\frac{\Gamma(k-\gamma\alpha)}
		{\Gamma(k+2+(\gamma-1)\alpha)}\\
		\leqslant {}&C_{\alpha,T} 
		\snm{y_0}^2
		\lambda^{\beta+\gamma}
		\sum_{k=M+1}^{\infty}
		k^{2\alpha-3-2\alpha(\beta+\gamma)}.
		\end{split}
		\]
		Take $\theta = (\beta+\gamma)/2$ and assume $1+2\alpha\theta>\alpha$, % there exist $-1\leqslant \beta\leqslant 1$ and $-1\leqslant \gamma\leqslant 1$ with   , %and assume $1+2\alpha\theta>\alpha$, 
		then using the proof of \cref{eq:est-L2}  implies the desired estimate \cref{eq:Phi-est-Hs}. %This ccompletes the proof.
	\end{proof}
One can observe that the key to get \cref{eq:Phi-est-Hs}  is to establish \cref{eq:est-Hs-1,eq:est-Hs-2}, which are easy to obtain for the singular function $y(t) = t^\gamma$. Hence, according to the proofs of \cref{lem:Phi-L2-w,lem:Phi-est-Hs}, it is not hard to conclude 
the following estimate.
	\begin{lem}\label{lem:Phi-est-tr}
		If $y(t) =t^\gamma$ with $\gamma>(\alpha-1)/2$, then %we have 
		\begin{align*}
			\nm{(I-\Phi_M^{-\alpha,0})y}_{L^{2}_{\mu^{-\alpha,0}}(0,T)} 
			+M^{-\alpha}		\snm{(I-\Phi_M^{-\alpha,0})y}_{H^{\alpha/2}(0,T)} 
			\leqslant {}&C_{\alpha,\gamma,T}
			M^{-1-2\gamma}.
		\end{align*}
	\end{lem}
	\begin{rem}
		We mention that  the optimal rate $1+2\gamma$ under $L^2$-norm has 
		already been proved in \cite[Theorem 5]{gui_h_1986_1} for $y(t) = t^\gamma$.
	\end{rem}
	For the particular case that $y_0=0$ and $g=1$, we can also establish optimal  bounds of projection error for the solution to the auxiliary problem \cref{eq:ode-y-f-0}. In fact, by the inequality \cref{eq:dk-y}, we are able to show that  \cref{eq:est-Hs-1,eq:est-Hs-2} also hold in this case.
	Since the proof techniques are almost the same as that of \cref{lem:Phi-L2-w,lem:Phi-est-Hs}, we omit the details here 
	and only list the main results as follows.
	\begin{lem}\label{lem:L2-y-A-1}
		Let  $y(t)$ be given by \cref{eq:y-y0-0-Besov} for $\lambda>0$, then
		\begin{equation}
		\label{eq:L2-y-A-1}
		\nm{\big(I-\Phi_M^{-\alpha,0}\big)y}_{L^2_{\mu^{-\alpha,0}}(0,T)}
		\leqslant {}C_{\alpha,T} \lambda^{\theta-1}M^{-1-2\theta\alpha},
		\end{equation}
		where $-1\leqslant \theta\leqslant 1$ and $1+2\alpha\theta>0$. Moreover,  
		\begin{equation}
		\label{eq:Hs-y-A-1}
		\snm{(I-\Phi_{M}^{-\alpha,0})y}_{H^{\alpha/2}(0,T)}
		\leqslant {}C_{\alpha,T} \lambda^{\theta-1}M^{\alpha-1-2\alpha\theta},
		\end{equation}
		where $-1\leqslant \theta\leqslant 1$ and $1+2\alpha\theta>\alpha$.
	\end{lem}
	%\begin{lem}\label{lem:Hs-y-A-1}
	%Let $\lambda>0$ and define $y(t)$ by \cref{eq:y-y0-0-Besov}, then
	%	\begin{equation}
	%		\label{eq:Hs-y-A-1}
	%	\snm{(I-\Phi_{M}^{-\alpha,0})y}_{H^{\alpha/2}(0,T)}
	%	\leqslant {}C_{\alpha,T} \lambda^{\theta-1}M^{\alpha-1-2\alpha\theta},
	%	\end{equation}
	%	where $-1\leqslant \theta\leqslant 1$ such that $1+2\alpha\theta>\alpha$.
	%\end{lem}	
	
	Below, we present a lemma that connects these projection error bounds above with our desired estimates. Recall that the space $\mathcal X$ 
	is given in \eqref{X-def} and that $\mathcal X^*$ denotes its dual space.
	\begin{lem}\label{lem:u-U}
		If $f+\D_{0+}^\alpha u_0\in\mathcal X^*$, then %we have
		\begin{align}
			\nm{u-U}_{L^2(0,T;\dot H^1(\Omega))}\lesssim{}&
			\nm{(I-R_h)u}_{\mathcal X} +
			\nm{(I- \Phi_{M}^{-\alpha,0})u}_{L^2(0,T;\dot H^1(\Omega))},
			\label{eq:err-L2H1}\\
			\nm{u-U}_{H^{\alpha/2}(0,T;L^2(\Omega))}\lesssim{}&
			\nm{(I-R_h)u}_{H^{\alpha/2}(0,T;L^2(\Omega))} +
			\nm{(I- \Phi_{M}^{-\alpha,0})u}_{\mathcal X}.
			\label{eq:err-HsL2}
		\end{align}
	\end{lem}
	\begin{proof}
		By \cref{eq:weak_form}, for any $V\in P_M(0,T)\otimes X_h$ we have
		\[
		\dual{\D_{0+}^\alpha u,V}_{H^{\alpha/2}(0,T;L^2(\Omega))} +
		\dual{\nabla u,\nabla V}_{\Omega\times(0,T)} =
		\dual{f+\D_{0+}^\alpha u_0,V}_{\mathcal X},
		\]
		which, together with \cref{eq:U}, gives the error equation
		\begin{equation}\label{eq:err}
		\dual{\D_{0+}^\alpha(u-U), V}_{H^{\alpha/2}(0,T;L^2(\Omega))} +
		\dual{\nabla(u-U),\nabla V}_{\Omega\times(0,T)}=0
		\quad\forall\, V\in P_M(0,T)\otimes X_h.
		\end{equation}
		Hence it follows that
		\[
		\begin{split}
		{}&\dual{\D_{0+}^\alpha (U-W),V}_{H^{\alpha/2}(0,T;L^2(\Omega))} +
		\dual{\nabla (U-W),\nabla V}_{\Omega\times(0,T)}  \\
		= {}&
		\dual{\D_{0+}^\alpha (u-W),V}_{H^{\alpha/2}(0,T;L^2(\Omega))} +
		\dual{\nabla (u-W),\nabla V}_{\Omega\times(0,T)} ,
		\end{split}
		\]
		where $W = \Phi_{M}^{-\alpha,0}R_hu$. Applying \cref{lem:Da} and the definition of the Ritz projection $R_h$
		%	Letting $W = \Phi_{M}^{-\alpha,0}R_hu$ and using 
		yields 
		\[
		\begin{split}
		{}&\dual{\D_{0+}^\alpha (U-W),V}_{H^{\alpha/2}(0,T;L^2(\Omega))} +
		\dual{\nabla (U-W),\nabla V}_{\Omega\times(0,T)}  \\
		= {}&
		\dual{\D_{0+}^\alpha (u-R_hu),V}_{H^{\alpha/2}(0,T;L^2(\Omega))} +
		\dual{\nabla (I- \Phi_{M}^{-\alpha,0})u,\nabla V}_{\Omega\times(0,T)},
		\end{split}
		\]
		and taking $V = U-W$ implies that
		\[
		\nm{U-W}_{\mathcal X}\lesssim
		\nm{(I-R_h)u}_{H^{\alpha/2}(0,T;L^2(\Omega))} +
		\nm{(I- \Phi_{M}^{-\alpha,0})u}_{L^2(0,T;\dot H^1(\Omega))}.
		%\begin{split}
		%{}&\snm{U- W}_{H^{\alpha/2}(0,T;L^2(\Omega))}+
		%\nm{U-W}_{L^2(0,T;\dot H^1(\Omega))}\\
		%\lesssim{}&
		%\snm{u- W}_{H^{\alpha/2}(0,T;L^2(\Omega))}+
		%\nm{u-W}_{L^2(0,T;\dot H^1(\Omega))},
		%\end{split}
		\]
		Now using  the triangle inequality and the stability result \cref{eq:stab-Phi-ma-0} gives
		\[
		\begin{split}
		{}&\nm{u-U}_{H^{\alpha/2}(0,T;L^2(\Omega))}
		\lesssim	\nm{u-W}_{H^{\alpha/2}(0,T;L^2(\Omega))}+\nm{U-W}_{H^{\alpha/2}(0,T;L^2(\Omega))}\\
		\lesssim{}&
		\nm{(I-  \Phi_{M}^{-\alpha,0})u}_{H^{\alpha/2}(0,T;L^2(\Omega))}+
		\nm{ \Phi_{M}^{-\alpha,0}(I-R_h)u}_{H^{\alpha/2}(0,T;L^2(\Omega))}+
		\nm{U-W}_{\mathcal X}\\
		\lesssim{}&\nm{(I-R_h)u}_{H^{\alpha/2}(0,T;L^2(\Omega))} +
		\nm{(I- \Phi_{M}^{-\alpha,0})u}_{\mathcal X}.
		\end{split}
		\]
		This establishes \cref{eq:err-HsL2}. As \cref{eq:err-L2H1} can be 
		proved similarly, we conclude the poof of this lemma. 
	\end{proof} 
	
	\subsection{Proofs of \cref{thm:est-u}--\ref{thm:est-f}}
	\medskip\noindent{\bf  Proof of \cref{thm:est-u}.} 
	Since $u(x,t) = y(t)\phi(x)$, where $\phi\in \dot H^2(\Omega)$ and $y(t) = t^\beta$ with $\beta>(\alpha-1)/2$, it is evident that
	\[
	\begin{split}
	\nm{(I-R_h)u}_{L^{2}(0,T;\dot H^1(\Omega))} 
	\lesssim {}&\nm{(I-R_h)\phi}_{\dot H^1(\Omega)} \lesssim 
	h \nm{\phi}_{\dot H^{2}(\Omega)},\\
	\nm{(I-R_h)u}_{H^{\alpha/2}(0,T;L^2(\Omega))} 
	\lesssim {}&\nm{(I-R_h)\phi}_{L^2(\Omega)} \lesssim 
	h^2 \nm{\phi}_{\dot H^{2}(\Omega)}.
	\end{split}
	\]
	In addition, applying \cref{lem:Phi-est-tr} gives 
	\[
	\begin{split}
	\nm{(I-\Phi_{M}^{-\alpha,0})u}_{L^{2}(0,T;\dot H^1(\Omega))} 
	\lesssim {}&\nm{\phi}_{\dot H^{1}(\Omega)}
	\nm{(I-\Phi_{M}^{-\alpha,0})y}_{L^{2}(0,T)} \lesssim 
	M^{-1-2\beta}\nm{\phi}_{\dot H^{1}(\Omega)},\\
	\nm{(I-\Phi_{M}^{-\alpha,0})u}_{H^{\alpha/2}(0,T;L^2(\Omega))} 
	\lesssim {}&\nm{\phi}_{L^{2}(\Omega)}
	\nm{(I-\Phi_{M}^{-\alpha,0})y}_{H^{\alpha/2}(0,T)} \lesssim 
	M^{\alpha-1-2\beta}\nm{\phi}_{L^{2}(\Omega)}.
	\end{split}
	\]
	Combining the above four estimates with \cref{lem:u-U} yields that
	\[
	\begin{split}
	\nm{u-U}_{L^{2}(0,T;\dot H^1(\Omega))} 
	\lesssim {}&\left(h+M^{-1-2\beta}\right) \nm{\phi}_{\dot H^{2}(\Omega)},\\
	\nm{u-U}_{H^{\alpha/2}(0,T;L^2(\Omega))} 
	\lesssim {}&\left(h^2+M^{\alpha-1-2\beta}\right) 
	\nm{\phi}_{\dot H^{2}(\Omega)},
	\end{split}
	\]
	which show \cref{eq:est-u-L2H1,eq:est-u-HsL2} and then finish the proof of \cref{thm:est-u}.
	\hfill\ensuremath{\blacksquare}
	%%%%%%%%%%Proof-of-Thm-4.2%%%%%%%%%%%%%%%%%%%%%%
	%%%%%%%%%%Proof-of-Thm-4.3%%%%%%%%%%%%%%%%%%%%%%
	\vskip 0.2cm
	Since the proof of \cref{thm:est-f} is parallel to that 
	of \cref{thm:est-u0}, we only consider the latter.
	
	\medskip\noindent{\bf  Proof of \cref{thm:est-u0}.} 
	According to \cref{thm:regu-u0}, we have
	%\[
	%\sqrt{\epsilon+\snm{2\alpha-1}}\nm{u}_{L^{2}(0,T;\dot H^{\min\{1/\alpha,2\}+\gamma-\epsilon}(\Omega))}  + 
	%\nm{u}_{H^{\alpha/2}(0,T;\dot H^{\min\{1/\alpha,2\}+\gamma-1}(\Omega))}  
	%\lesssim \nm{u_0}_{\dot H^{\gamma}(\Omega)},
	%\]
	%where $\epsilon\in[0,1]$ such that $\epsilon+\snm{2\alpha-1}>0$. Thus, it follows that 
	\[
	\begin{split}
	\nm{(I-R_h)u}_{L^{2}(0,T;\dot H^1(\Omega))} 
	\lesssim {}&
	\left\{
	\begin{aligned}
	&\frac{1}{\sqrt{\epsilon}}h^{\min\{1,\gamma+1\}-\epsilon}
	\nm{u_0}_{\dot H^{\gamma}(\Omega)},&&\alpha=1/2,\\
	&h^{\min\{1,\gamma+1,1/\alpha+\gamma-1\}}
	\nm{u_0}_{\dot H^{\gamma}(\Omega)},&&\alpha\neq1/2,
	\end{aligned}
	\right.\\
	\nm{(I-R_h)u}_{H^{\alpha/2}(0,T;L^2(\Omega))} 
	\lesssim {}&
	h^{\min\{2,\gamma+1,1/\alpha+\gamma-1\}}
	\nm{u_0}_{\dot H^{\gamma}(\Omega)}.
	\end{split}
	\]
	For $\alpha=1/2$, we choose $\epsilon = 1/(2+\snm{\ln h})$ to get
	\[
	\nm{(I-R_h)u}_{L^{2}(0,T;\dot H^1(\Omega))} 
	\lesssim {}
	\sqrt{\snm{\ln h}}\,h^{\min\{1,\gamma+1\}}
	\nm{u_0}_{\dot H^{\gamma}(\Omega)}.
	\]
	If $\max\{-1,1-\frac1\alpha\}< \gamma\leqslant 3$, then invoking \cref{lem:decomp-u,lem:Phi-L2-w,rem:L2} gives the estimate
	\[
	\nm{(I-\Phi_{M}^{-\alpha,0})u}_{L^{2}(0,T;\dot H^1(\Omega))} 
	\lesssim {}
	M^{-1-\alpha(\gamma-1)}\nm{u_0}_{\dot H^{\gamma}(\Omega)}.
	\]
	And if $\max\{-1,1-\frac1\alpha\}<  \gamma\leqslant 2$, then by \cref{lem:decomp-u,rem:L2,eq:L2-y-A-1} we conclude that
	\[
	\nm{(I-\Phi_{M}^{-\alpha,0})u}_{\mathcal X} 
	\lesssim M^{-1-\alpha(\gamma-1)}
	\nm{u_0}_{\dot H^{\gamma}(\Omega)}.
	\]
	Consequently, combining the above estimates with \cref{lem:u-U} 
	proves \cref{eq:est-u0-HsL2,eq:est-u0-L2H1} and thus
	%\[
	%\begin{split}
	%\nm{u-U}_{L^{2}(0,T;\dot H^1(\Omega))} 
	%\lesssim{}& {}\left(\epsilon_hh^{\min\{1,\gamma+1,1/\alpha+\gamma-1\}} +M^{-1-\alpha(\gamma-1)}\right) \nm{u_0}_{\dot H^{\gamma}(\Omega)},\\
	%\nm{u-U}_{H^{\alpha/2}(0,T;L^2(\Omega))} 
	%\lesssim {}&\left(h^{\min\{2,\gamma+1,1/\alpha+\gamma-1\}}
	%+M^{-1-\alpha(\gamma-1)}\right) 
	%\nm{u_0}_{\dot H^{\gamma}(\Omega)}.
	%\end{split}
	%\]
	%\[
	%\nm{u-U}_{H^{\alpha/2}(0,T;L^2(\Omega))} 
	%\lesssim {}\left(h^{2+\gamma}+M^{-1-\alpha(\gamma+1)}\right) 
	%\nm{v}_{\dot H^{\gamma}(\Omega)}.
	%\]
	%By \cref{rem:L2,lem:u-sum,lem:Phi-L2-w,lem:Phi-est-Hs}, we conclude that
	%\[
	%\begin{split}
	%\nm{(I-\Phi_{M}^{-\alpha,0})u}_{L^{2}(0,T;\dot H^1(\Omega))} 
	%\lesssim {}&
	%M^{-1-2\alpha\theta}\nm{u_0}_{\dot H^{2\theta+1}(\Omega)},\\
	%\nm{(I-\Phi_{M}^{-\alpha,0})u}_{H^{\alpha/2}(0,T;L^2(\Omega))} 
	%\lesssim {}&
	%M^{\alpha-1-2\alpha\theta}\nm{u_0}_{\dot H^{2\theta}(\Omega)},
	%\end{split}
	%\]
	%where $-1\leqslant \theta\leqslant 1$ such that $1+2\alpha\theta>0$.  Collecting the above estimates and using \cref{lem:u-U} yield the desired results:
	%\[
	%\begin{split}
	%\nm{u-U}_{L^{2}(0,T;\dot H^1(\Omega))} 
	%\lesssim {}&\left(\epsilon_hh^{\min\{1/\alpha,2\}+\gamma-1} +M^{-1-\alpha(\gamma-1)}\right) \nm{u_0}_{\dot H^{\gamma}(\Omega)},\quad \theta = \frac{\gamma-1}{2}\\
	%\nm{u-U}_{H^{\alpha/2}(0,T;L^2(\Omega))} 
	%\lesssim {}&\left(h^{\min\{1/\alpha,2\}+\gamma-1}
	%+M^{-1-\alpha(\gamma-1)}\right) 
	%\nm{u_0}_{\dot H^{\gamma}(\Omega)}.
	%\quad \theta = \frac{\gamma}{2}
	%\end{split}
	%\]
	%This proves \cref{eq:est-u0-L2H1,eq:est-u0-HsL2} and 
	completes the proof of \cref{thm:est-u0}.
	\hfill\ensuremath{\blacksquare}
	\section{Numerical Tests}
	\label{sec:numer}
	This section presents several numerical examples to validate our theoretical predictions. For simplicity, we take $ T = 1 $, $ \Omega = (0,1) $, and set
	\begin{align*}
		%    \mathcal E_1 & := \nm{\widehat u - U}_{L^2(0,T;L^2(\Omega))},\\
		\mathcal E_1 & := \nm{\widehat u - U}_{L^2(0,T;\dot H^1(\Omega))},\\
		\mathcal E_2 & := \nm{\widehat u - U}_{H^{\alpha/2}(0,T;L^2(\Omega))},
	\end{align*}
	where $ \widehat u $ is the reference solution in the case of $M = 150,\,h = 2^{-10}$. 
	
	As the spatial discretization utilizes the standard conforming finite element method, which has been investigated and confirmed in \cite{li_luo_xie_analysis_2019}, in 
	the following we are mainly interested in the convergence behavior of temporal discretization errors.
	\begin{eg}
		\label{eg:true-u}
		This example is to verify \cref{thm:est-u}
		with an a priorly known solution 
		\[
		u(x,t)  := t^\beta\sin\pi x,\quad (x,t)\in\Omega\times(0,T),
		\]
		where $\beta>(\alpha-1)/2$. Temporal discretization errors are plotted in \cref{fig:temp_test_1}, where the following convergence rates are observed:
		\[
		\mathcal E_1 = O(M^{-1-2\beta}),\quad \mathcal E_2 = O(M^{\alpha-1-2\beta}).
		\]
		These agree well with the theoretical results given by \cref{thm:est-u}.
		\begin{figure}[H]
			\centering
			\includegraphics[width=440pt,height=180pt]{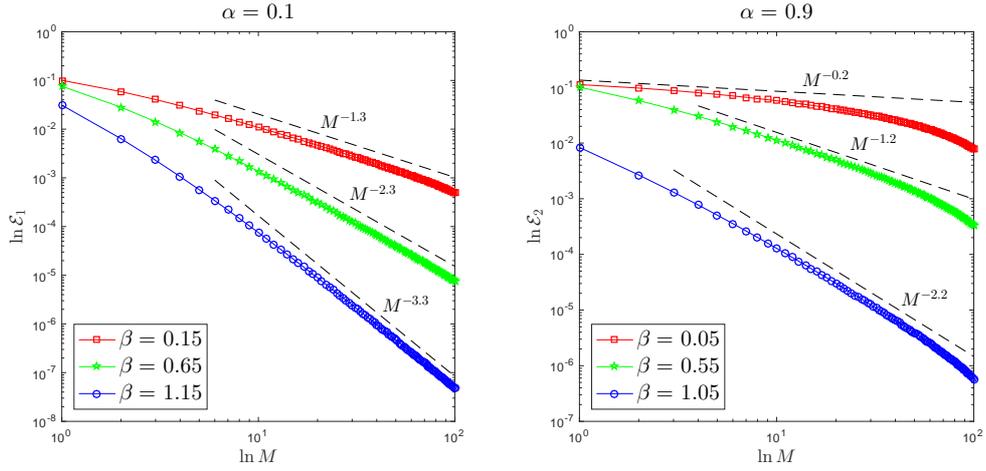}
			\caption{Temporal errors of \cref{eg:true-u} with $h = 2^{-10}$.}
			\label{fig:temp_test_1}
		\end{figure}
	\end{eg}
	\begin{eg}\label{eg:u0}
		To verify \cref{thm:est-u0}, we consider $f= 0$ and
		\[
		u_0(x)  :=\theta x(1-x)^{\gamma-1/2}+(1-\theta)\sin\pi x,\quad x\in\Omega,
		\]
		where $\max\{-1,1-\frac1\alpha\}<\gamma<5/2$ and $\theta=0,1$. For $\theta = 1$, a straightforward calculation yields that $u_0\in \dot H^{\gamma-\epsilon}(\Omega)$ for any $\epsilon>0$; and for $\theta = 0$, we have $u_0 = \sin\pi x\in\dot H^{\beta}(\Omega)$ with any $\beta>0$, since $\sin\pi x$ is an eigenfunction of $-\Delta $ on $\Omega=(0,1)$ with the homogeneous Dirichlet boundary condition. The convergence behaviour is plotted in \cref{fig:temp_test_2}, which implies that
		\[
		\begin{split}
		\mathcal E_1 ={}& 
		\left\{
		\begin{aligned}
		&O(M^{-1-\alpha(\gamma-1)}),&& -1/2<\gamma<3,\,\theta = 1,\\
		&O(M^{-1-2\alpha}),&& \theta = 0,
		\end{aligned}
		\right.\\
		\mathcal E_2 ={}& 
		\left\{
		\begin{aligned}
		&O(M^{-1-\alpha(\gamma-1)}),&& -0.1<\gamma<2,\,\theta = 1,\\
		&O(M^{-1-\alpha}),&& \theta = 0.
		\end{aligned}
		\right.
		\end{split}
		\]
		These coincide with the sharp estimates established in \cref{thm:est-u0}.  
		\begin{figure}[H]
			\centering
			\includegraphics[width=440pt,height=180pt]{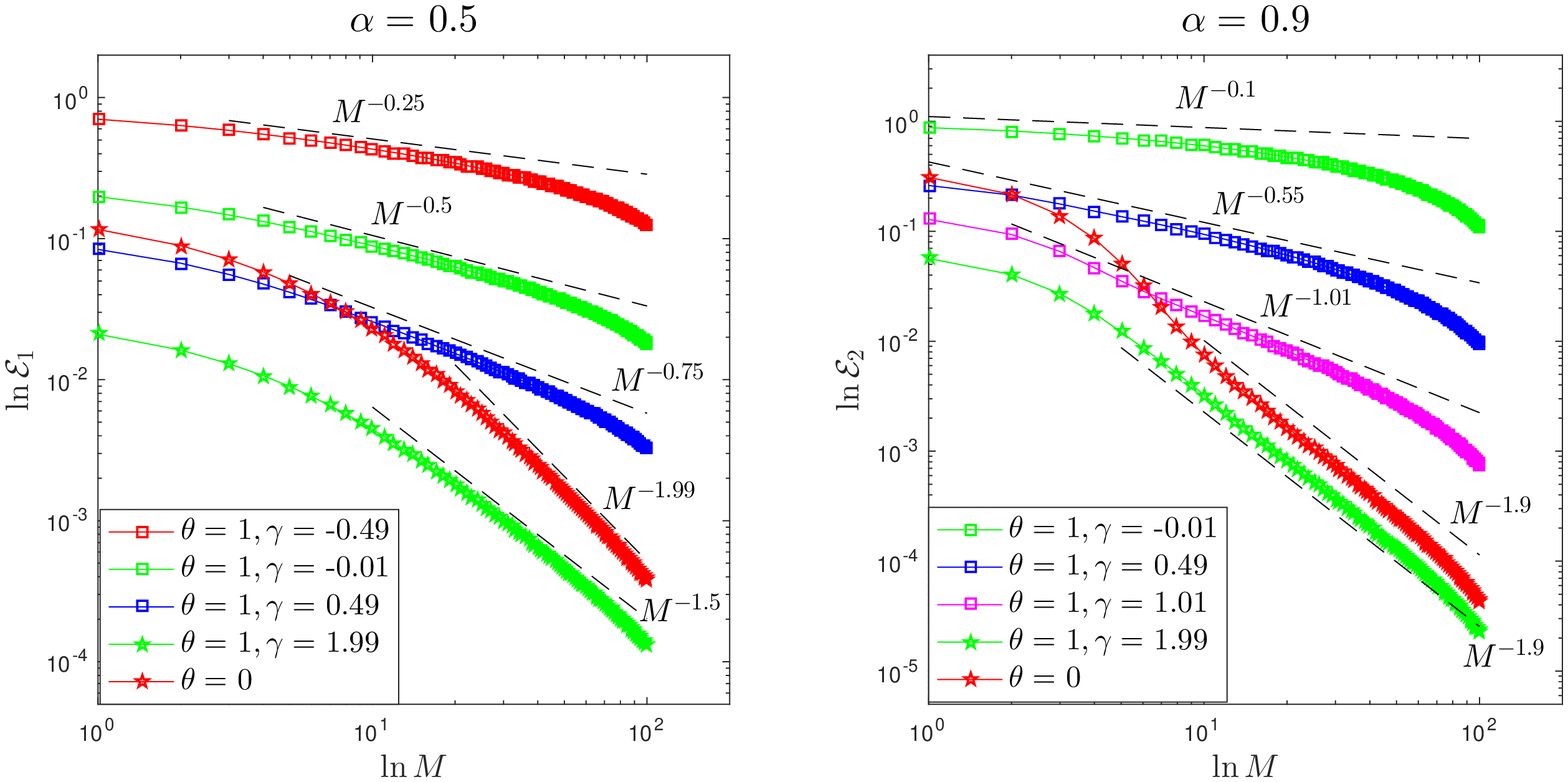}
			\caption{Temporal errors of \cref{eg:u0} with $h = 2^{-10}$.}
			\label{fig:temp_test_2}
		\end{figure}
	\end{eg}
	\begin{eg}	\label{eg:f}
		This example  is to verify
		\cref{thm:est-f} with $u_0 = 0$ and 
		\[
		f(x,t) = v(x) := x^{\gamma-1/2}(1-x),\quad x\in\Omega,
		\]	
		where $-1/2<\gamma<3/2$. It is evident that $v\in \dot H^{\gamma-\epsilon}(\Omega)$ for any $\epsilon>0$. Numerical results are displayed in \cref{fig:temp_test_3}, from which we  conclude that 
		\[
		\begin{split}
		\mathcal E_1 ={}& 
		\left\{
		\begin{aligned}
		&O(M^{-1-\alpha(\gamma+1)}),&& -1/2<\gamma<1,\\
		&O(M^{-1-2\alpha}),&& 1\leqslant \gamma<3/2,
		\end{aligned}
		\right.\\
		\mathcal E_2 ={}& 
		\left\{
		\begin{aligned}
		&O(M^{-1-\alpha(\gamma+1)}),&& -1/2<\gamma<0,\\
		&O(M^{-1-\alpha}),&& 0\leqslant \gamma<3/2.
		\end{aligned}
		\right.
		\end{split}
		\]
		These are conformable to  the sharp error bounds presented in \cref{thm:est-f}.
		\begin{figure}[H]
			\centering
			\includegraphics[width=440pt,height=180pt]{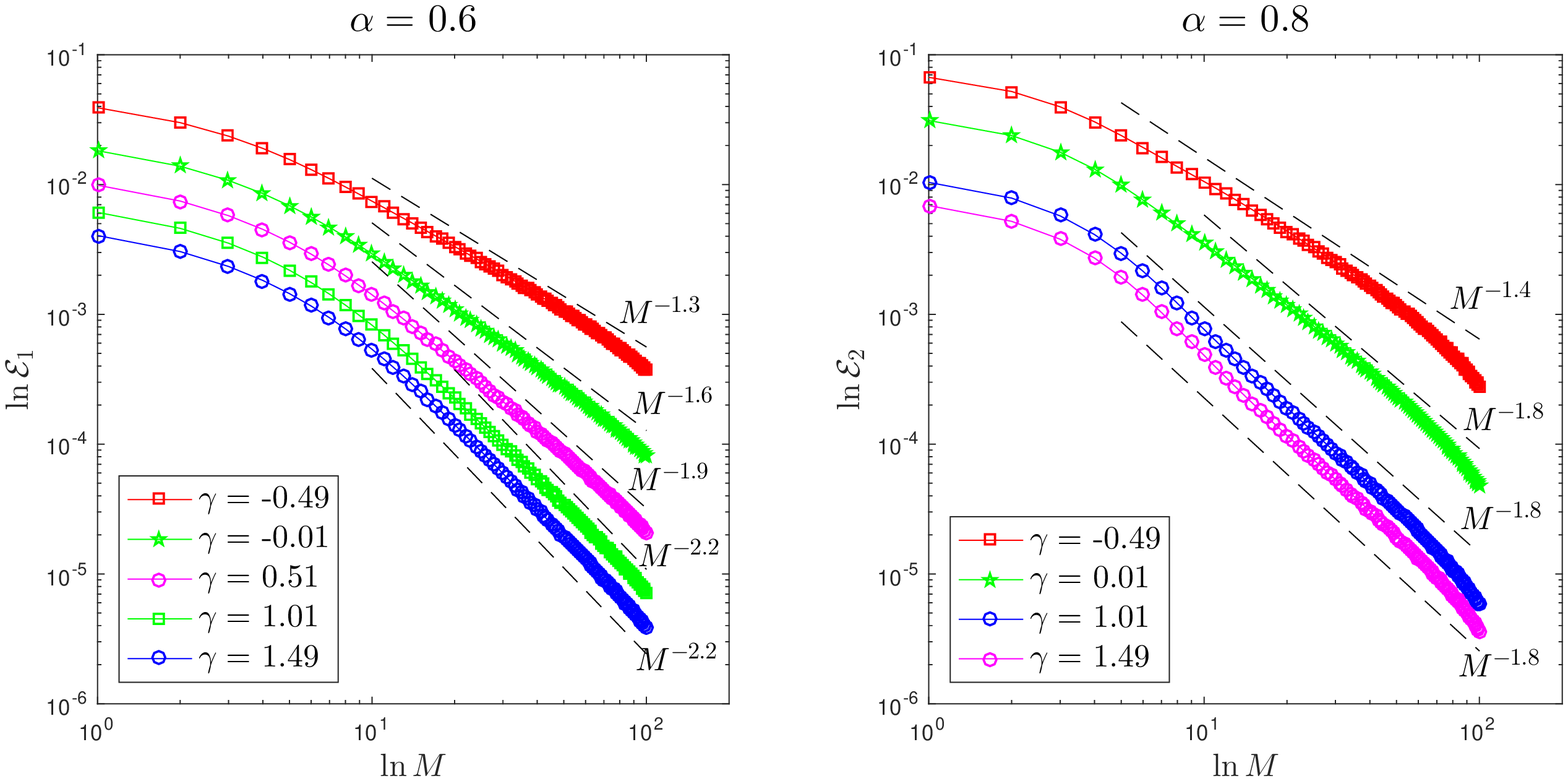}
			\caption{Temporal errors of \cref{eg:f} with $h = 2^{-10}$.}
			\label{fig:temp_test_3}
		\end{figure}
	\end{eg}
	%\[
	%\dual{x(1-x),\sin n\pi x}_{(0,1)} = \frac{4{\sin^2\left(\frac{\pi\, n}{2}\right)}}{n^3\pi^3} = 
	%\frac{2(\cos n\pi-1)}{n^3\pi^3} 
	%\]
	%\[
	%x(1-x) = \sum_{n=0}^{\infty}c_n\phi_n = \sum_{n=0}^{\infty}c_n\sin n\pi x,
	%\]
	%\[
	%c_n = 
	%\frac{\dual{x(1-x),\sin n\pi x}_{(0,1)}}{\nm{\phi_n}^2_{L^2(0,1)}} = 2\dual{x(1-x),\sin n\pi x}_{(0,1)}
	%=\frac{4(\cos n\pi-1)}{n^3\pi^3} 
	%\]
	%\[
	%\sum_{n=0}^{\infty}c_n^2\lambda_n^\gamma = \sum_n n^{2\gamma}c_n^2
	%\]
	%\[
	%c_n = \dual{x^{-0.99}(1-x),\sin n\pi x}_{(0,1)} \sim n^{-.0.01}
	%\quad\Longrightarrow \dot H^{-1/2}(\Omega)
	%\]
	%\[
	%c_n = \dual{x^{-1/2}(1-x),\sin n\pi x}_{(0,1)} \sim n^{-1/2}
	%\quad\Longrightarrow L^2(\Omega)
	%\]
	%\[
	%c_n = \dual{x^{0.01}(1-x),\sin n\pi x}_{(0,1)} \sim n^{-1}
	%\quad\Longrightarrow \dot H^{1/2}(\Omega)
	%\]
	%\[
	%c_n = \dual{x^{1/2}(1-x),\sin n\pi x}_{(0,1)} \sim n^{-1.5}
	%\quad\Longrightarrow \dot H^1(\Omega)
	%\]
	%\[
	%c_n = \dual{x^{1.01}(1-x),\sin n\pi x}_{(0,1)} \sim n^{-2}
	%\quad\Longrightarrow \dot H^{3/2}(\Omega)
	%\]
	%\[
	%c_n = \dual{x^{1.51}(1-x),\sin n\pi x}_{(0,1)} \sim n^{-2.5}
	%\quad\Longrightarrow \dot H^2(\Omega)
	%\]
	%\[
	%c_n = \dual{x(1-x),\sin n\pi x}_{(0,1)} \sim n^{-3}
	%\quad\Longrightarrow \dot H^{5/2}(\Omega)
	%\]
	%\input{6_concl}
	\section{Conclusion}
	\label{sec:concl}
	This paper has concerned the sharp error estimation of the time-spectral algorithm for time
	fractional diffusion problems of order $\alpha$ ($0 < \alpha < 1$). Based on the new regularity results in the Besov space, optimal convergence rates of the  numerical algorithm have been derived with low regularity data. Particularly, sharp temporal convergence orders $1+2\alpha$ and $1+\alpha$ under $L^2(0,T;\dot H^1(\Omega))$-norm and $H^{\alpha/2}(0,T;L^2(\Omega))$-norm have been shown theoretically and  numerically. 
	
	Beyond the presented results, several problems are deserving further investigation. The first is to establish the error estimate under $L^2(0,T;L^2(\Omega))$-norm. We conjecture that, {\it for both the homogenuous case that $f=0,\,u_0\in\dot H^{2}(\Omega)$ and the nonhomogeneous case that $u_0=0,\,f = v\in L^2(\Omega)$, the theoretical temporal accuracy under $L^2(0,T;L^2(\Omega))$-norm shall be $1+2\alpha$}, which has been verified via our numerical tests (not displayed in the context). The second is to conquer the boundary singularity via some singular basis or log orthogonal function \cite{chen_log_2020} and establish a rigorous error analysis for nonsmooth data. 
	
	\appendix
	
	\section{The Shifted Jacobi Polynomial}
	\label{sec:sJacobi-append}
	Given $a,b>-1$, the family of shifted Jacobi polynomial $\{S_k^{a,b}\}_{k=0}^\infty$ on $(0,T)$ are defined as follows:
	\begin{equation}\label{eq:St}
	\mu^{a,b}(t)S_{k}^{a,b}(t)=\frac{(-1)^{k}}{T^{k} k !}
	\frac{\dd^{k}}{\dd t^{k}}\mu^{k+a,k+b}(t),\quad 0<t<T,
	\end{equation}
	where $\mu^{\nu,\theta}(t)=(T-t)^{\nu}t^{\theta}$ for all $-1<\nu,\theta<\infty$. 
	Note that \cref{eq:St} is also called Rodrigues' formula \cite{shen_spectral_2011}, which implies $\{S_k^{a,b}\}_{k=0}^\infty$ 
	is orthogonal with respect to the weight $\mu^{a,b}$ on $(0,T)$, i.e.,
	\begin{equation}\label{eq:orth-S}
	\dual{S_{k}^{a, b},S_{l}^{a, b}}_{\mu^{a,b}}
	=\xi_k^{a, b} \delta_{kl},
	\end{equation}
	where $\delta_{kl}$ denotes the Kronecker product and 
	\begin{equation}\label{eq:xiab}
	\xi_k^{a, b}:=\frac{T^{a+b+1} \Gamma(k+a+1) \Gamma(k+b+1)}{(2k+a+b+1) k! \Gamma(k+a+b+1)}.
	\end{equation}
	
	As $\{S_k^{a,b}\}_{k=0}^\infty$ forms a complete orthogonal basis of $L^2_{\mu^{a,b}}(0,T)$, any $v\in L^2_{\mu^{a,b}}(0,T)$ admits a unique decomposition
	\[
	v=\sum_{k=0}^{\infty}
	v_kS_k^{a,b} \quad\text{with } v_k = \frac{1}{\xi_k^{a,b}}
	\dual{v,S_{k}^{a, b}}_{\mu^{a,b}},
	\]
	and the $L^2_{\mu^{a,b}}$-orthogonal projection 
	of $v$ onto $P_M(0,T)$ is defined as 
	\[
	\Phi_{M}^{a,b}v:=
	\sum_{k=0}^{M}
	v_kS_k^{a,b}.
	\]
	For ease of notation, we shall set $S_k^a = S_k^{a,a},\,\mu^{a} = \mu^{a,a},\,\Phi_M^a = \Phi_M^{a,a}$, and all the superscripts are omitted 
	when $a = 0$. 

	%
	%\begin{lem}
	%\label{lem:I}
	%	If $\rho>0$ and $\beta>-1$, then
	%\[
	%\D_{1-}^{-\rho}P_{n}^{(0, \beta)}(x)=\frac{\Gamma(n+1)}{\Gamma(n+\rho+1)}(1-x)^{\rho} P_{n}^{(\rho, \beta-\rho)}(x).
	%\]
	%	If $\rho>0$ and $\alpha>-1$, then
	%\[
	%\D_{-1+}^{-\rho}P_{n}^{(\alpha, 0)}(x)=\frac{\Gamma(n+1)}{\Gamma(n+\rho+1)}(1+x)^{\rho} P_{n}^{(\alpha-\rho, \rho)}(x).
	%\]
	%	\end{lem}
	%\begin{lem}
	%	If $\rho>0$ and $\beta>-1$, then
	%	\[
	%	\D_{1-}^{\rho}\left\{
	%	(1-x)^{\rho} P_{n}^{(\rho, \beta-\rho)}(x)\right\}
	%	=\frac{\Gamma(n+\rho+1)}{\Gamma(n+1)}
	%	P_{n}^{(0, \beta)}(x).
	%	\]
	%	If $\rho>0$ and $\alpha>-1$, then
	%	\[
	%	\D_{-1+}^{\rho}\left\{
	%	(1+x)^{\rho} P_{n}^{(\alpha-\rho, \rho)}(x)
	%	\right\}
	%	=\frac{\Gamma(n+\rho+1)}{\Gamma(n+1)}
	%	P_{n}^{(\alpha, 0)}(x).
	%	\]
	%\end{lem}
	Thanks to \cite[Lemma 2.4]{chen_generalized_2014}, a standard calculation gives
	%Hence, for the shifted Jacobi polynomial $\{S_k^{a,b}(t)\}_{k=0}^\infty$, we have
	\begin{equation}\label{eq:IsS}
	\D_{0+}^{-\theta}
	S_k^{\beta+\theta,0} 
	={} \frac{ k!}{\Gamma(k+1+\theta)}t^{\theta}
	S_k^{\beta,\theta}(t),\quad 
	\D_{T-}^{-\theta}	S_k^{0,\beta+\theta} 
	={} \frac{ k!}{\Gamma(k+1+\theta)}
	(T-t)^{\theta}S_k^{\theta,\beta}(t)
	\end{equation}
	and 
	\begin{equation}\label{eq:DsS}
	\D_{0+}^{\theta}
	S_k^{\beta-\theta,0} 
	={} \frac{ k!}{\Gamma(k+1-\theta)}t^{-\theta}
	S_k^{\beta,-\theta}(t),\quad
	\D_{T-}^{\theta}	S_k^{0,\beta-\theta} 
	={} \frac{ k!}{\Gamma(k+1-\theta)}
	(T-t)^{-\theta}
	S_k^{-\theta,\beta}(t),
	\end{equation}
	where $0<\theta<1$ and $-1<\beta<\infty$.
	%Below, let us present a useful lemma.
	\begin{lem}
		\label{lem:Da}
		For any $v\in H^{\alpha/2}(0,T)$, we have 
		\begin{align}
			\label{eq:orth-D0+}
			\dual{\D_{0+}^{\alpha}(I-\Phi_{M}^{-\alpha,0})v,q}_{H^{\alpha/2}(0,T)} = 0
			\quad \forall\,q\in P_M(0,T).
			%		\label{eq:orth-DT-}
			%		\dual{\D_{T-}^{\alpha}(I-\Phi_{M}^{0,-\alpha})v,q}_{H^{\alpha/2}(0,T)} = 0
			%		\quad \forall\,q\in P_M(0,T).
		\end{align}
		Consequently, it holds the stability  
		\begin{align}
			\label{eq:stab-Phi-ma-0}
			\snm{\Phi_{M}^{-\alpha,0}v}_{H^{\alpha/2}(0,T)} \leqslant {}&
			C_{\alpha}\snm{v}_{H^{\alpha/2}(0,T)}.
			%		\label{eq:stab-Phi-0-ma}
			%	\snm{\Phi_{M}^{0,-\alpha}v}_{H^{\alpha/2}(0,T)} \leqslant {}&
			%	C_{\alpha}\snm{v}_{H^{\alpha/2}(0,T)}.
		\end{align}
	\end{lem}
	\begin{proof}
		%	Since the proof of \cref{eq:orth-DT-} is anlaougous to that of \cref{eq:orth-D0+}, we only focus on the latter. 
		Given any $v\in H^{\alpha/2}(0,T)$, by \cite[Theorem 1.4.4.3]{grisvard_elliptic_nodate} it is clear that $v\in L_{\mu^{-\alpha,0}}^2(0,T)=\bb_{-\alpha,0}^0(0,T)$. Thus we have the orthogonal decomposition 
		\[
		v=	\sum_{k=0}^{\infty}
		v_kS_k^{-\alpha,0},\quad v_k = \frac{1}{\xi_k^{-\alpha,0}}
		%	\int_{0}^{T} v(t)S_{k}^{-\alpha,0}(t) 
		%	\mu^{-\alpha,0}(t) \dd t
		\dual{v,S_k^{-\alpha,0}}_{	\mu^{-\alpha,0}},
		%	\xi_k^{-\alpha,0},
		%	=\frac{T^{1-\alpha} }{2k+1 -\alpha},
		\]
		where $	\xi_k^{-\alpha,0}$ is given by \cref{eq:xiab},
		and the $L^2_{\mu^{-\alpha,0}}$-orthogonal projection of $v$ onto $P_M(0,T)$ is given by
		\[
		\Phi_M^{-\alpha,0}v=
		\sum_{k=0}^{M}
		v_kS_k^{-\alpha,0}.
		\]
		Hence, the projection error reads as  
		\[
		(I-\Phi_M^{-\alpha,0})v=
		\sum_{k=M+1}^{\infty}
		v_kS_k^{-\alpha,0}.
		\]	
		For any $q\in P_M(0,T)$, we rewrite it as an expansion of $\{S_k^{0,-\alpha}\}_{k=0}^M$:
		\[
		q=
		\sum_{k=0}^{M}
		q_kS_k^{0,-\alpha},\quad q_k = \frac{1}{\xi_k^{0,-\alpha}}
		\dual{q,S_k^{0,-\alpha}}_{	\mu^{0,-\alpha}},
		%	\xi_k^{0,-\alpha},
		%	=\frac{T^{1-\alpha} }{2k+1 -\alpha}.
		\]
		where $	\xi_k^{0,-\alpha}$ is defined by \cref{eq:xiab}. 
		From \cref{eq:DsS} we get
		\[
		\D_{0+}^{\alpha/2}(I-\Phi_{M}^{-\alpha,0})v
		%	= \sum_{k=M+1}^{\infty}y_k
		%	\D_{0+}^{\alpha/2}S_k^{-\alpha,0}
		=
		\sum_{k=M+1}^{\infty}
		\frac{	v_kk!	}{\Gamma(k+1-\alpha/2)}
		t^{-\alpha/2}
		S_k^{-\alpha/2},
		\]
		%and
		\[
		\D_{T-}^{\alpha/2}q 
		%	= \sum_{k=0}^{M}q_k
		%	\D_{T-}^{\alpha/2}J_k^{0,-\alpha}
		=
		\sum_{k=0}^{M}
		\frac{	q_kk!	}{\Gamma(k+1-\alpha/2)}
		(T-t)^{-\alpha/2}
		S_k^{-\alpha/2}.
		\]
		Applying \cref{lem:coer} yields the equality
		\[
		\dual{\D_{0+}^{\alpha}(I-\Phi_{M}^{-\alpha,0})v,q}_{H^{\alpha/2}(0,T)} 
		=	
		\dual{\D_{0+}^{\alpha/2}(I-\Phi_{M}^{-\alpha,0})v,\D_{T-}^{\alpha/2}q}_{(0,T)},
		\]
		and it follows from the orthogonality of $\{S^{-\alpha/2}_k\}_{k=0}^\infty$ 
		with respect to the weight $\mu^{-\alpha/2}$ that
		\[
		\dual{\D_{0+}^{\alpha/2}(I-\Phi_{M}^{-\alpha,0})v,\D_{T-}^{\alpha/2}q}_{(0,T)} = 0\quad\forall\,q\in P_M(0,T).
		\]
		This shows \cref{eq:orth-D0+}. As \cref{eq:stab-Phi-ma-0} is a direct corollary of \cref{eq:orth-D0+}, we complete the proof of this lemma.
	\end{proof}
	\begin{rem}
		Define the generalized Jacobi orthogonal projection 
		\[
		\pi_M^{-\alpha}:L^2_{\mu^{-\alpha}}(0,T)
		\longrightarrow \{t^{\alpha}\}\otimes P_M(0,T)
		\]
		by that 
		\[
		\dual{(I-		\pi_M^{-\alpha})v,q}_{\mu^{-\alpha}} = 0
		\quad\forall\,q\in \{t^{\alpha}\}\otimes P_M(0,T)\quad  \forall v\in L^2_{\mu^{-\alpha}}(0,T).
		\]
		%		\[
		%\int_0^T\big(v(t)-(\pi_M^{-\alpha,-\alpha}v)(t)\big) q(t)\mu^{-\alpha,-\alpha}(t)\dd t  =0 \quad\forall\,q\in \{t^{\alpha}\}\otimes P_M(0,T),
		%		\]
		%where $v\in L^2_{\mu^{-\alpha}}(0,T)$.
		Then for any $v\in H^{\alpha/2}(0,T)$, we have an identity similar with \cref{eq:orth-D0+} (see \cite[Remark 4.3]{chen_generalized_2014} or \cite[Eq. 3.26]{shen_efficient_2019}):
		\[
		\dual{\D_{0+}^{\alpha}(I-\pi_{M}^{-\alpha})v,q}_{H^{\alpha/2}(0,T)} = 0
		\quad \forall\,q\in P_M(0,T).
		\]
	\end{rem}
	\bibliographystyle{abbrv}
	%\bibliography{/Users/luohao/Desktop/GitHub/Mymicro/mylibrary}
	
\end{document}